\documentclass[sn-mathphys,Numbered,iicol]{sn-jnl}

\usepackage{graphicx}%
\usepackage{multirow}%
\usepackage{amsmath,amssymb,amsfonts}%
\usepackage{amsthm}%
\usepackage{mathrsfs}%
\usepackage[title]{appendix}%
\usepackage{xcolor}%
\usepackage{textcomp}%
\usepackage{manyfoot}%
\usepackage{booktabs}%
\usepackage{algorithm}%
\usepackage{algorithmicx}%
\usepackage{algpseudocode}%
\usepackage{listings}%
\usepackage{enumerate}
\usepackage{subcaption}

\newtheorem{theorem}{Theorem}
\newtheorem{proposition}{Proposition}%
\newtheorem{lemma}{Lemma}
\newtheorem{assumption}{Assumption}%

\newcommand{\vertiii}[1]{{\left\vert\kern-0.25ex\left\vert\kern-0.25ex\left\vert #1 
    \right\vert\kern-0.25ex\right\vert\kern-0.25ex\right\vert}}

\begin{document}

\title[Random primal-dual algorithms for parallel MRI]{On the convergence and sampling of randomized primal-dual algorithms and their application to parallel MRI reconstruction}



\author*[1]{\fnm{Eric B} \sur{Gutierrez}}\email{ebgc20@bath.ac.uk}

\author[2]{\fnm{Claire} \sur{Delplancke}}\email{cd902@bath.ac.uk}

\author[1]{\fnm{Matthias J} \sur{Ehrhardt}}\email{M.Ehrhardt@bath.ac.uk}

\affil[1]{\orgdiv{Department of Mathematical Sciences}, \orgname{University of Bath}, \orgaddress{\postcode{BA2 7AY}, \country{UK}}}

\affil[2]{\orgdiv{EDF Lab Paris-Saclay},\orgaddress{ \street{route de Saclay}, \postcode{91300 Palaiseau}, \country{France}}}


\abstract{Stochastic Primal-Dual Hybrid Gradient (SPDHG) is an algorithm proposed by Chambolle et al. (2018)\ to efficiently solve a wide class of nonsmooth large-scale optimization problems. In this paper we contribute to its theoretical foundations and prove its almost sure convergence for convex but neither necessarily strongly convex nor smooth functionals, as well as for any random sampling. 
In addition, we {study} SPDHG for parallel Magnetic Resonance Imaging reconstruction, where data from different coils are randomly selected at each iteration. We apply SPDHG using a wide range of random sampling methods {and} compare its performance across a range of settings, including mini-batch size and step size parameters. 
We show that the sampling can significantly affect the convergence speed of SPDHG {and} for many cases an optimal sampling can be identified. }


\maketitle

\section{Introduction}\label{sec:intro}

Inverse problems can be solved using variational regularization, generally presented as an optimization problem~\cite{benning2018modern}. In fields such as imaging, data science or machine learning, optimization challenges are  generally formulated as the convex minimization problem
\begin{equation} \label{min}
    \hat{x} \in \arg\min_{x\in X} \sum_{i=1}^n f_i(A_ix) + g(x) 
\end{equation}
where $f_i:Y_i\to\mathbb{R}\cup\{\infty\}$ and $g:X\to\mathbb{R}\cup\{\infty\}$ are convex functionals, and $A_i:X\to Y_i$ are linear and bounded operators between real Hilbert spaces.

Many active areas of research exists within this framework, with problems such as regularized risk minimization~\cite{boserGuyonVapnik, cevherBigData,shalevZhang,zhangXiao}, heavily-constrained optimization~\cite{fercoq_etal2019,patrascuNecoara2017} or total variation regularization image reconstruction~\cite{ROF}.
Within the context of image reconstruction, examples of problems in the form of~(\ref{min}) are image denoising~\cite{chambolle2016introduction}, PET reconstruction~\cite{ehrhardt2019faster} and, most relevant to the present paper, parallel Magnetic Resonance Imaging (MRI) reconstruction~\cite{mri,hager2015alternating,pruessmann2006encoding}. 

Problem~(\ref{min}) makes no assumptions on the differentiability of the convex functionals $f_i$, $g$. In general, if these functionals are not smooth then many classical approaches such as gradient descent are not applicable \cite{chambolle2016introduction}. In contrast, primal-dual methods are able to solve (\ref{min}) even for non-smooth functionals. When $f_i$, $g$ are convex, proper and lower-semicontinuous, the \textit{saddle point problem} for (\ref{min}) is
\begin{equation} \label{saddle}
    \hat{x},\hat{y} \in \arg\min_{x\in X}\max_{y\in Y} L(x,y)
\end{equation}
with $y:=(y_1,...,y_n)$, $Y:=\Pi_{i=1}^n Y_i$ and
$$ L(x,y) := \sum_{i=1}^n \{ \langle A_ix,y_i \rangle - f_i^*(y_i) \} + g(x) $$
where $f_i^*$ denotes the \textit{convex conjugate} of $f_i$. We refer to a solution of~(\ref{saddle}) as a \emph{saddle point}.

A well-known primal-dual method is the Primal-Dual Hybrid Gradient algorithm (PDHG) \cite{chambollePock,esserZhangChan,CremersChambollePock}.
While PDHG is proven to converge to a saddle point, its iterations can be costly for large-scale problems, e.g.~when $n\gg1$ \cite{spdhg}.
A random primal-dual method, the Stochastic Primal-Dual Hybrid Gradient algorithm (SPDHG) was proposed recently by Chambolle et al.~\cite{spdhg}. Its main difference over PDHG is that it reduces the per-iteration computational cost by randomly sampling the dual variable, i.e.~only a random subset of the dual variable gets updated at every iteration. In~\cite{spdhg}, it is shown that SPDHG offers significantly better performance than the deterministic PDHG for large-scale problems \cite{ehrhardt2019faster,Schramm_Holler}. 
Examples of other random primal-dual algorithms can be found in
\cite{fercoqBianchi2019,gao2019randomPD,latafat2019randomPD,zhangXiao}.

In the original paper~\cite{spdhg}, SPDHG's almost sure convergence is proven for strongly convex functionals $f^*_i, g$. Later, for the special case of {serial sampling}, Alacaoglu et al.~proved SPDHG's almost sure convergence for arbitrary convex functionals in finite-dimensional Hilbert spaces~\cite{alacaoglu}. An alternative proof can also be found in~\cite{gutierrez}.

In this paper we complete the gap on the convergence theory of SPDHG and prove its almost sure convergence for any convex functionals on Hilbert spaces and for any arbitrary random sampling.

This generalizes the existing convergence results in two ways. The first is the extension into Hilbert spaces, 
which are important in many areas such as quantum mechanics and partial differential equations~\cite{sakurai2014quantum}.
%

The second generalization is the convergence of SPDHG for any arbitrary sampling. The proofs for convergence presented in both~\cite{alacaoglu} and~\cite{gutierrez} rely heavily on the assumption that SPDHG uses serial sampling for choosing the random updates. In this paper we present a proof of convergence for any sampling, and we propose adequate step size conditions for several specific examples of random samplings.

Furthermore, as a novel application of SPDHG we perform parallel MRI reconstruction on real data. Parallel MRI constitutes an inverse problem with a data structure that can be naturally randomized, and serves as an ideal candidate to compare performance of SPDHG under different random samplings.
Our numerical examples show that information on the forward problem can help identify an optimal sampling as well as optimal step size parameters.

\section{Stochastic Primal-Dual Hybrid Gradient}
\label{sec:alg}

Before presenting the main convergence results, we formally introduce SPDHG, which is the result of randomizing the dual variable in the deterministic PDHG method. PDHG with dual extrapolation~\cite{chambollePock, pock2011diagonal} applied to Problem~(\ref{min}) is
\begin{equation} \label{det}
\begin{aligned}
    x^{k+1} &=\mathrm{prox}_{\tau g}(x^{k} - \tau A^*\bar{y}^{k}) \\
    y_i^{k+1} &= 
    \mathrm{prox}_{\sigma f_i^{*}}(y_i^{k}+\sigma_i A_ix^{k+1})
    \text{ for all } i \\
    \bar{y}^{k+1} &= y^{k+1} + \theta(y^{k+1} - y^{k})
\end{aligned}
\end{equation}
%
%
where $\tau,\sigma >0$ are step-size parameters, $\bar{y}^{k+1}$ is an extrapolation with parameter $\theta\in\mathbb{R}$, operator $A^*:Y\to X$ is $A^{*}y=\sum_iA^*_iy_i$,
and the \textit{proximity operator}~\cite{bauschkeCombettes}~is
\begin{equation*}
    \mathrm{prox}_{h}(v) :=  \arg\min_{u} \left\{ \frac{\|v-u\|^2}{2} + h(u) \right\}.
\end{equation*}
This method is proven to converge for $\theta=1$ and step size condition $\tau\sigma\|A^*\|^2<1$ for every $i\in\{1,...,n\}$, for any convex functionals $g,f^*_i$ \cite{chambollePock}. 

SPDHG, in contrast, reduces the cost of (\ref{det}) by updating only a random subset of the coordinates $(y_i)_{i=1}^n$.
This means,
at every iteration $k$, a subset $\mathbb{S}^k\subset\{1,...,n\}$ is chosen at random
and only the variables $y_{i}^{k+1}$ for $i\in\mathbb{S}^k$ are updated, while the rest remain unchanged:
\begin{equation*}
    y_{i}^{k+1}  = \left\{ 
    \begin{tabular}{ll}
    $\mathrm{prox}_{\sigma_i f_{i}^{*}}(y_{i}^{k}+\sigma_{i} A_{i}x^{k+1})$ \quad &  if $i\in\mathbb{S}^{k}$ \\
    $y_{i}^{k}$ & else.
    \end{tabular} \right.
\end{equation*}
Here we allow parameter $\sigma_i$ to be different for each $i$. Furthermore, 
we assume the random variables $\mathbb{S}^k$ are independent and identically distributed, and that the sampling must be {proper}, i.e.\ it must satisfy
\begin{equation}\label{pi}
    p_i:=\mathbb{P}(i\in\mathbb{S}^k)>0 \quad \textrm{for every } i\in\{1,...,n\},
\end{equation}
meaning the probability $p_i$ of updating any coordinate~$i$ must be nonzero.
The complete SPDHG algorithm is given in Algorithm~\ref{alg_2}.

\begin{algorithm} 
\caption{SPDHG}
\label{alg_2} 
\begin{algorithmic}
\State Choose $x^0\in X$ and $y^0\in Y$.
\State Set $z^0=\bar{z}^0=A^*y^0$. 
\For{$k\geq0$}
\State select $\mathbb{S}^{k} \subset\{1,...\,,n\}$ at random
\State $x^{k+1} =\mathrm{prox}_{\tau g}(x^{k} - \tau\bar{z}^k)$
\State $y_{i}^{k+1}  = \left\{ 
        \begin{tabular}{ll}
        $\mathrm{prox}_{\sigma_i f_{i}^{*}}(y_{i}^{k}+\sigma_{i} A_{i}x^{k+1})$ & \hspace{-4.5mm} if $i\in\mathbb{S}^{k}$ \\
        $y_{i}^{k}$ & \hspace{-4.5mm} else
        \end{tabular} \right.$
\State \hspace{3.5mm} $\delta_{i}  = A^*_{i}(y_{i}^{k+1}-y_{i}^k)$ \hspace{14.5mm} for all $i\in\mathbb{S}^{k}$
\State $z^{k+1} = z^k + \sum_{i\in\mathbb{S}^k}\delta_i$
\State $\bar{z}^{k+1} = z^{k+1} + \theta\sum_{i\in\mathbb{S}^k}p_{i}^{-1}\delta_i $
\EndFor
\end{algorithmic}
\end{algorithm}

In order to minimize the number of linear operations, the auxiliary variable $z^k$ stores the current value of $A^*y^k$, so that only the operators $A^*_i$ for $i\in\mathbb{S}^k$ need to be evaluated at each iteration. 
Similarly, $\bar{z}^{k+1}$ represents the extrapolation
\begin{equation*}
    \bar{z}^{k+1} = A^*y^{k+1} + A^*Q(y^{k+1}-y^k)
\end{equation*}
where $Q:Y\to Y$ is the operator defined by
$ (Qy)_i=p_i^{-1}y_i .$
We will often refer to the fact that $Q$ is symmetric and positive definite.

For {serial sampling}, where only one coordinate is selected at every iteration, i.e. $|\mathbb{S}^k|=1$ for every $k$, SPDHG takes the form of Algorithm~\ref{alg_spdhg}, which is the special case of Algorithm~\ref{alg_2} where we set $\mathbb{S}^k=\{j^k\}$ and only $y_{j^k}^{k+1}$ and $\delta = \delta_{j^k}$ are activated at each iteration $k$.

\begin{algorithm} 
\caption{SPDHG for serial sampling}
\label{alg_spdhg} 
\begin{algorithmic} 
\State Choose $x^0\in X$ and $y^0\in Y$.
\State Set $z^0=\bar{z}^0=A^*y^0$. 
\For{$k\geq0$}
\State select $j^{k} \in\{1,...\,,n\}$ at random
\State $x^{k+1} =\mathrm{prox}_{\tau g}(x^{k} - \tau\bar{z}^k)$
\State $y_{i}^{k+1}  = \left\{ 
        \begin{tabular}{ll}
        $\mathrm{prox}_{\sigma_i f_{i}^{*}}(y_{i}^{k}+\sigma_{i} A_{i}x^{k+1})$ \quad & \hspace{-5.5mm} if $i=j^{k}$ \\
        $y_{i}^{k}$ & \hspace{-5.5mm}  else
        \end{tabular} \right.$
\State \hspace{4.5mm} $\delta = A^*_{j^k}(y_{j^k}^{k+1}-y_{j^k}^k)$
\State $z^{k+1} = z^k + \delta$
\State $\bar{z}^{k+1} = z^{k+1} + \theta p_{j^k}^{-1}\delta $
\EndFor
\end{algorithmic}
\end{algorithm}

\section{Convergence of SPDHG}
\label{sec:main}
In this section we state the almost sure convergence of SPDHG to a saddle point of the primal-dual problem~(\ref{saddle}), as well as the conditions under which this is true. This result extends the usefulness of SPDHG by generalizing its existing convergence theory into any arbitrary sampling, as well as into Hilbert spaces. 

Hilbert spaces are relevant in imaging since 2D images can be modelled by square-integrable functions on the plane~\cite{bredieslorenz}.
Examples of imaging problems in Hilbert spaces are wavelet shrinkage denoising, where the challenge consists in finding a sparse representation of the image in the infinite basis~\cite{vetterli1995wavelets}; or super-resolution imaging, which describes the inverse problem of finding the square-integrable function that best approximates a set of low-resolution images~\cite{chaudhuri2001super}.


\begin{assumption} \label{assu}
Problem~(\ref{saddle}) satisfies:
\begin{enumerate}
    \item The Hilbert spaces $X,Y$ are real and separable. \label{assu0}
    \item The set of solutions to~(\ref{saddle}) is nonempty. \label{assu1}
    \item The functionals $g,f_i$ are convex, proper and lower-semicontinuous. \label{assu2}
    \item The proximity operators $\mathrm{prox}_{\tau g}$, $\mathrm{prox}_{\sigma_i f_i^*}$ are weakly sequentially continuous. \label{assu3}
\end{enumerate}
\end{assumption}

Assumptions (\ref{assu0}), (\ref{assu1}) and (\ref{assu2}) are as in the original SPDHG result (\cite{spdhg}, Theorem~4.3), and
Assumption~(\ref{assu1}) is further discussed in~\cite{alacaoglu}.
Assumption (\ref{assu3}) is always satisfied in finite dimensions, since the proximity operator $\mathrm{prox}_h$ is continuous for any convex, proper and lower-semicontinuous functional $h$ (\cite{bauschkeCombettes}, Proposition 12.28).

In general, a function $T$ is \textit{weakly sequentially continuous} if, for every sequence $(v^k)_{k\in\mathbb{N}}$ converging weakly to $v$, the sequence $(T(v^k))_{k\in\mathbb{N}}$ converges weakly to $T(v)$.{ For examples of separable sums of convex functions that satisfy Assumption~(\ref{assu3}), see (\cite{bauschkeCombettes}, Proposition~24.12). }

{We recall that a weak \textit{cluster point} of a sequence $(x^k)_{k\in\mathbb{N}}$ is any point $\hat{x}$ such that there exists a subsequence $(x^{\ell_k})_{k\in\mathbb{N}}$ which converges weakly to $\hat{x}$. Furthermore, for a probability space denoted by $(\mathbf{\Omega},\mathcal{F},\mathbb{P})$, we say a sequence $(x^k)_{k\in\mathbb{N}}$ converges weakly \textit{almost surely} to $\hat{x}$ if there exists $\Omega\in\mathcal{F}$ such that $\mathbb{P}(\Omega)=1$ and $x^k(\omega)\to\hat{x}(\omega)$ weakly for every $\omega\in\Omega$.
For brevity, we write \textit{a.s.} instead of {almost surely}.
}

Additionally, in order to state the step size conditions, we use the following notation. For any subset $\mathcal{S}\subset\{1,...,n\}$, let $A_\mathcal{S}:X\to Y$ be given by 
\begin{equation*}
    (A_\mathcal{S}x)_i := \left\{ \begin{tabular}{ll}
        $A_ix$ & if $i\in\mathcal{S}$ \\
        $0$ & else
    \end{tabular} \right. 
\end{equation*}
%
with adjoint 
$A_\mathcal{S}^*y  = \sum_{i\in\mathcal{S}} A_i^*y_i$. 

Similarly, denote $C_i = \tau^{1/2}\sigma^{1/2}_i A_i$ for every $i\in\{1,..,n\}$, with adjoints $C_i^*=\tau^{1/2}\sigma_i^{1/2}A^*_i$, and $C_\mathcal{S}$ and $C_\mathcal{S}^*$ defined accordingly as above.
For any random sampling $\mathbb{S}$, we define the step size operator $D$ as 
\begin{equation} \label{B}
    D := Q\mathbb{E}\left(C_\mathbb{S}C_\mathbb{S}^*\right)Q.
\end{equation}

\begin{theorem}[\textbf{Convergence of SPDHG}]
\label{main}
Let Assumption~\ref{assu} be satisfied, $\theta=1$ and
\begin{equation} \label{assu_i}
    \|D\| < 1.
\end{equation}
Then for any proper sampling, {SPDHG}
(Algorithm~\ref{alg_2}) converges weakly a.s.\ to a solution of~(\ref{saddle}).
\end{theorem}

\subsection{Step size condition}
In the original paper of SPDHG \cite{spdhg}, the authors describe the step size condition using \textit{ESO parameters} \cite{quartz}. While these conditions are equivalent, we prefer to use~(\ref{assu_i}) because it offers a practical way to check the validity of the step-size parameters $\tau,\sigma_i$ by finding the eigenvalues of $D$. Moreover, from \cite{quartz} we know $D:Y\to Y$ is symmetric and positive semi-definite and can be expressed as
\begin{equation} \label{PCC}
    D = \left( \begin{tabular}{ccc}
        $D_{11}$ & $\cdots$ & $D_{1n}$ \\
        $\vdots$ & $\ddots$ & $\vdots$ \\
        $D_{n1}$ & $\cdots$ & $D_{nn}$ \\
    \end{tabular} \right)
\end{equation}
where $D_{ij}:Y_j\to Y_i$ are given by
\begin{equation*}
    D_{ij} = \frac{p_{ij}}{p_ip_j} C_iC_j^*
\end{equation*}
and $p_{ij} = \mathbb{P}\left(i\in\mathbb{S},j\in\mathbb{S}\right)$.
In contrast, ESO parameters are parameters $v_1,...,v_n$ such that 
\begin{equation} \label{ESO}
    \mathbb{E} \left\| C_{\mathbb{S}}^*z \right\|^2 \leq \sum_{i=1}^n p_iv_i\|z_i\|^2
\end{equation}
and the step size condition for SPDHG in~\cite{spdhg} is 
\begin{equation} \label{vi}
    v_i < p_i \quad \textrm{ for all } i\in\{1,...,n\}.
\end{equation}
The following lemma shows this condition is equivalent to~(\ref{assu_i}). 

\begin{lemma} \label{BESO}
Let $D$ be defined as in~(\ref{B}). Then $\|D\|<1$ if and only if there exist ESO parameters $v_i$ such that $v_i<p_i$ for $i\in\{1,...,n\}$.
\end{lemma}

\begin{proof} 
Let $\|D\|<1$. Then
\begin{equation*} 
\begin{aligned}
    \mathbb{E} \|C^*_\mathbb{S} z\|^2 
    & = \langle z,\mathbb{E}(C_\mathbb{S}C_\mathbb{S}^*)z \rangle  = \langle Q^{-1}z,DQ^{-1}z \rangle \\ & \leq \|D\| \|Q^{-1}z\|^2 = \|D\| \sum_{i=1}^n p_i^2 \|z_i\|^2.
\end{aligned}
\end{equation*}
hence~(\ref{ESO}) and~(\ref{vi}) are satisfied by choosing
$$v_i = \|D\|p_i.$$
Conversely, let $v_i$ satisfy~(\ref{ESO}) and~(\ref{vi}), then
\begin{equation*}
\begin{aligned}
    \langle Dz,z \rangle &= \langle \mathbb{E}\left( C_\mathbb{S}C_\mathbb{S}^* \right)Qz , Qz\rangle = \mathbb{E} \|C^*_\mathbb{S} Qz\|^2 \\ & \leq \sum_{i=1}^n p_iv_i\|p_i^{-1} z_i\|^2 \\ & < \sum_{i=1}^n p_i^2\|p_i^{-1} z_i\|^2 = \|z\|^2
\end{aligned}
\end{equation*}
and using the fact that $D$ is symmetric and positive semidefinite, the left-hand side becomes $\langle Dz,z \rangle = \langle{D}^{1/2}z,D^{1/2} z\rangle = \|D^{1/2}z\|^2 $, which proves $\|D^{1/2}\|<1$, and thus $\|D\|<1$.
\end{proof}

Clearly, the step size parameters $\tau,\sigma_i$ that satisfy~(\ref{assu_i}) are not unique. In particular, if we assume the step size parameters to be uniform, i.e.~$\sigma_i=\sigma$ for all $i$, then we write
$ D = \tau\sigma Q\mathbb{E}(A_{\mathbb{S}}A_\mathbb{S}^*)Q $
and thus it suffices to choose $\tau,\sigma$ such that
\begin{equation} \label{tausig}
    \tau\sigma \|Q\mathbb{E}(A_{\mathbb{S}}A_\mathbb{S}^*)Q\| < 1.
\end{equation}
%

In Section~\ref{sec:step}, we will see examples on how to find optimal step sizes that comply with this condition for specific types of random samplings.

\subsection{Proof of Theorem~\ref{main}} \label{sec:ske}

The two following propositions lay out the proof of Theorem~\ref{main}. We use the notation $w^k = (x^k,y^k)$.

\begin{proposition} \label{tozero1}
Let $\theta=1$ and $(w^k)_{k\in\mathbb{N}}$ a random sequence generated by Algorithm~\ref{alg_2} under Assumption~\ref{assu} and step size condition~(\ref{assu_i}). The following assertions hold:
\begin{enumerate}[i]
    \item \label{i} The sequence $(w^{k+1}-w^{k})_{k\in\mathbb{N}}$ converges a.s.\ to zero.
    \item \label{iii} The sequence $(\|w^k-\hat{w}\|)_{k\in\mathbb{N}}$ converges a.s.\ {(not necessarily to zero)} for every saddle point $\hat{w}$.
    \item \label{iv}  If every weak cluster point of $(w^k)_{k\in\mathbb{N}}$ is a.s.~a saddle point, the sequence $(w^k)_{k\in\mathbb{N}}$ converges weakly a.s.\ to a saddle point. 
\end{enumerate}
\end{proposition}

The key idea of the proof consists in rewriting Algorithm~\ref{alg_2} as a sequence of operators $T_\mathcal{S}$ which depend on the random subsets $\mathcal{S}$. 
To show this, denote 
\begin{equation*}
    w=(w_0,w_1,...,w_n)=(x,y_1,...,y_n)
\end{equation*}
and, for every set $\mathcal{S}\subset \{1,...,n\}$, let the operator $T_\mathcal{S}:X\times Y\to X\times Y$ be defined by
\begin{align*}
&(T_\mathcal{S}w)_0  =\\
& \mathrm{prox}_{\tau g} \Big( x-\tau A^*y - 
\sum_{j\in\mathcal{S}} \frac{1+p_j}{p_j} 
\tau A_j^*((T_\mathcal{S} w)_j-y_j) \Big)
\end{align*}
and, for $i\in\{1,...,n\}$,
\begin{equation*} 
(T_\mathcal{S}w)_i = \begin{cases}
\mathrm{prox}_{\sigma_i f_i^*}(y_i+\sigma_iA_ix)
& \text{if } i\in\mathcal{S} \\
y_i & \text{else}. 
\end{cases}
\end{equation*}

\begin{proposition}
\label{cluster_saddle}
Let $\theta=1$ and $(w^k)_{k\in\mathbb{N}}$ a random sequence generated by Algorithm~\ref{alg_2} under Assumption~\ref{assu} and step size condition~(\ref{assu_i}). The following assertions hold:
\begin{enumerate}[i]
    \item \label{2i}
    The iterates $w^k=(x^k,y^k)$ satisfy, for every $k$,
    \begin{equation} \label{itTs}
        T_{\mathbb{S}^k}(x^{k+1},y^{k}) = (x^{k+2},y^{k+1}) .
    \end{equation}
    \item \label{2ii}
    A point $\hat{w}\in X\times Y$ is a saddle point if and only if it is a fixed point of $T_{\mathbb{S}^k}$ for every instance of $\mathbb{S}^k$.
    \item \label{2iii}
    Every weak cluster point of $(w^k)_{k\in\mathbb{N}}$ is a.s.~a saddle point.
\end{enumerate}
\end{proposition}

\begin{proof}[Proof of Theorem \ref{main}]
By Proposition~\ref{cluster_saddle}-\ref{2iii}, every weak cluster point of $(w^k)_{k\in\mathbb{N}}$  is almost surely a saddle point and, by Proposition~\ref{tozero1}-\ref{iv}, the sequence $(w^k)_{k\in\mathbb{N}}$ converges weakly almost surely to a saddle point. 
\end{proof}

\subsection{Proof of Propositions \ref{tozero1} \& \ref{cluster_saddle}} \label{sec:conv}

We use the notation $\|x\|^2_T=\langle Tx,x\rangle$ for any operator $T$. 
We also denote 
\begin{equation*}
    V(x,y) := \|x\|^{2}_{\tau^{-1}} + 2\langle QAx,y \rangle + \|y\|^{2}_{QS^{-1}}    
\end{equation*}
with $S=\mathrm{diag}(\sigma_1,...,\sigma_n)$ and $A$ given by $(Ax)_i=A_ix$. For any
{$\hat{w}=(\hat{x},\hat{y})$} we write
%
%
\begin{equation*} 
    \Delta^k_{\hat{w}} 
    := V(x^k-\hat{x},y^{k-1}-y^k) + \|y^k-\hat{y}\|^{2}_{QS^{-1}}
\end{equation*}
and the conditional expectation at time $k$ is denoted as 
$\mathbb{E}^{k}(w) = \mathbb{E}(w \,|\, w^0,...,w^{k-1})$.
With this notation we will make use of (\cite{alacaoglu}, Lemma~4.1), which was originally obtained as a consequence of~(\cite{spdhg}, Lemma~4.4).

\begin{lemma}[\cite{alacaoglu}, Lemma 4.1]
\label{ala}
Let $(w^k)_{k\in\mathbb{N}}$ be a random sequence generated by Algorithm~\ref{alg_2} under Assumption~\ref{assu}. Then for every saddle point $\hat{w}$, 
\begin{equation} \label{inequality0}
\begin{aligned}
    \Delta^k_{\hat{w}} \geq &\mathbb{E}^{k+1}(\Delta^{k+1}_{\hat{w}}) \\ &+ V(x^{k+1}-x^k,y^k-y^{k-1}). 
\end{aligned}
\end{equation}
\end{lemma}

In~(\ref{inequality0}) we have simplified the original lemma by using the fact that Bregman distances of convex functionals are nonnegative \cite{spdhg}. We will also require the following lemma, which is equivalent to (\cite{spdhg}, Lemma~4.2) with $\rho^2=\max_i{v_i}/{p_i}$ and $c=\rho^{-1}$, and where we have replaced the ESO step size condition~(\ref{vi}) with the equivalent condition~(\ref{assu_i}).

\begin{lemma}[\cite{spdhg}, Lemma 4.2]
\label{lemma4.2}
Let $D$ be defined as in~(\ref{B}) such that $\|D\|=\rho^2<1$ and let $(y^k)_{k\in\mathbb{N}}$ be obtained through Algorithm \ref{alg_2}. Then for all $x\in X$, $k\in\mathbb{N}$,
\begin{equation*}
\begin{aligned}
    &\mathbb{E}^k V(x,y^k-y^{k-1}) \\  &\geq (1-\rho) \mathbb{E}^k\big(\|x\|^{2}_{\tau^{-1}} + \|y^k-y^{k-1}\|^{2}_{QS^{-1}} \big).
\end{aligned}
\end{equation*}
\end{lemma}

\begin{proof}[Proof of Proposition \ref{tozero1}-\ref{i}]
Let $d^k$ denote
$$ d^k := \|x^{k+1}-x^{k}\|^{2}_{\tau^{-1}} + \|y^{k}-y^{k-1}\|^{2}_{QS^{-1}}. $$
Taking the expectation in~(\ref{inequality0}) and applying Lemma~\ref{lemma4.2} yields
\begin{equation*}
\begin{aligned}
    \mathbb{E}(\Delta^k_{\hat{w}}) \geq \mathbb{E}(\Delta^{k+1}_{\hat{w}}) 
    + (1-\rho)\mathbb{E}(d^k) .
\end{aligned}
\end{equation*}
Taking the sum from $k=0$ to $k=N-1$ gives 
\begin{equation} \label{sum}
\begin{aligned}
    \Delta_{\hat{w}}^0 \geq \mathbb{E}(\Delta^{N}_{\hat{w}}) 
    + (1-\rho)\sum_{k=0}^{N-1} \mathbb{E}(d^k)
\end{aligned}
\end{equation}
where we have used $y^{-1}:=y^0$. From Lemma~\ref{lemma4.2} we know 
$\mathbb{E}(\Delta^N_{\hat{w}})\geq0$, hence taking the limit as $N\to\infty$ in (\ref{sum}) yields $\sum_{k=0}^\infty \mathbb{E}(d^k) < \infty$. By the monotone convergence theorem, this is equivalent to 
\begin{equation} \label{ntozero}
     \mathbb{E}\Big( \sum_{k=0}^\infty  
     d^k
     \Big) < \infty,
\end{equation}
which implies
\begin{equation}\label{sum(y^n)} 
    \sum_{k=0}^\infty 
    d^k
    < \infty \;\textrm{a.s.}
\end{equation} 
and it follows that
\begin{equation}\label{yto0}
   d^k
   \to0 \text{ a.s.}
\end{equation}
Since the norms $\|\cdot\|_{\tau^{-1}}$ and $\|\cdot\|_{QS^{-1}}$ are equivalent to the usual norms in $X$ and $Y$, we deduce the sequence $(w^{k+1}-w^k)_{k\in\mathbb{N}}$ converges to zero almost surely.
\end{proof}

For the next part we require a slight variation of Lemma~\ref{lemma4.2}.

\begin{lemma} \label{n/c}
Let $\varphi = \max_{i}{\|C_i\|}/{\sqrt{p_i}}$ and let $(y^k)_{k\in\mathbb{N}}$ be the iterates defined by Algorithm~1. For all $c>0$, $x\in X$ and $k\in\mathbb{N}$, there holds
\begin{equation*}
\begin{aligned}
    &V(x,y^k-y^{k-1}) \\ &\geq (1-\frac{n\varphi}{c}) \|x\|_{\tau^{-1}}^2 + (1-\varphi c)\|y^{k}-y^{k-1}\|_{QS^{-1}}^2 .
\end{aligned}
\end{equation*}
\end{lemma}

\begin{proof}
~
\begin{equation*}
\begin{aligned}
    &\langle QAx, y^k-y^{k-1} \rangle \\ &= 
    \sum_{i\in{\mathbb{S}^k}} \langle p_i^{-1}A_ix, y_i^k-y_i^{k-1} \rangle \\
    &= \sum_{i\in{\mathbb{S}^k}} \langle p_i^{-1/2}C_i\tau^{-1/2}x,p_i^{-1/2}\sigma_i^{-1/2}(y_i^k-y_i^{k-1}) \rangle
    \\
    &\leq \sum_{i\in{\mathbb{S}^k}} \|p_i^{-1/2}C_i\|\|x\|_{\tau^{-1}}\|y_i^k-y^{k-1}_i\|_{p_i^{-1}\sigma_i^{-1}}
    \\
    &\leq \max_{i}\frac{\|C_i\|}{\sqrt{p_i}} \sum_{i\in{\mathbb{S}^k}} \frac1{2c}\|x\|_{\tau^{-1}}^2 + \frac{c}{2}\|y_i^k-y^{k-1}_i\|_{p_i^{-1}\sigma_i^{-1}}^2  \\
    &\leq \max_{i}\frac{\|C_i\|}{\sqrt{p_i}} \frac{1}{2} \Big( \frac{n}{c}\|x\|_{\tau^{-1}}^2 + c\|y^k-y^{k-1}\|_{QS^{-1}}^2 \Big).
\end{aligned}
\end{equation*}
\end{proof}

We also recall a classical result by Robbins \& Siegmund. 

\begin{lemma}[\cite{robbins1971}, Theorem 1] 
\label{RobbSieg}
Let $\mathcal{F}_k$ be a sequence of sub-$\sigma$-algebras such that, for every $k$, $\mathcal{F}_k\subset \mathcal{F}_{k+1}$ and $\alpha_k$, $\eta_k $ are nonnegative $\mathcal{F}_k$-measurable random variables such that 
$$\sum_{k=1}^\infty \eta_k <\infty \quad \mathrm{a.s.}$$
and
\begin{equation*}
    \mathbb{E}(\alpha_{k+1}\,|\,\mathcal{F}_k) \,\leq\, \alpha_k + \eta_k \quad\textrm{a.s.}
\end{equation*}
Then $(\alpha_k)_{k\in\mathbb{N}}$ converges almost surely to a random variable in $[0,\infty)$.
\end{lemma}

\begin{proof}[Proof of Proposition \ref{tozero1}-\ref{iii}]

Let $c=(n+1)\varphi$. By Lemma~\ref{n/c}, for any saddle point $\hat{w}$
%
%
%
\begin{equation} \label{Delta>0}
\begin{aligned}
    \Delta^k_{\hat{w}} \geq & \frac{1}{n+1}\|x^k-\hat{x}\|_{\tau^{-1}}^2  + \|y^k-\hat{y}\|_{QS^{-1}}^2 \\ & + (1-(n+1)\varphi^2)\|y^{k}-y^{k-1}\|_{QS^{-1}}^2.
\end{aligned}
\end{equation}

Since $y^k-y^{k-1}\to0$ a.s.~this implies $\Delta^k_{\hat{w}}$ is a.s. bounded from below,
i.e.~there exists a random variable $M_1\geq0$ (independent of $k$) such that $ \alpha^k := \Delta^k_{\hat{w}}+M_1 \geq 0$ a.s.~for every $k$. Let $\eta^k:=2|\langle QA(x^{k+1}-x^{k}),y^{k}-y^{k-1} \rangle|$. Then by (\ref{inequality0}), 
\begin{equation} \label{supermartingale}
    \alpha^k + \eta^k \,\geq\, \mathbb{E}^{k+1}(\alpha^{k+1}) \quad\textrm{a.s. for every } k
\end{equation}
where all terms are nonnegative and,  for some $M_2\geq0$,
\begin{equation*}
\begin{aligned}
    \eta^k &= 2|\langle QA(x^{k+1}-x^{k}),y^{k}-y^{k-1} \rangle| \\
    &\leq 2\|Q A\| \|x^{k+1}-x^k\| \|y^k-y^{k-1}\| \\
    &\leq 2M_2  \|x^{k+1}-x^k\|_{\tau^{-1}} \|y^k-y^{k-1}\|_{QS^{-1}} \\
    &\leq M_2 \big( \|x^{k+1}-x^k\|_{\tau^{-1}}^2 + \|y^k-y^{k-1}\|^2_{QS^{-1}} \big) 
\end{aligned}
\end{equation*}
which together with~(\ref{sum(y^n)}) implies $\sum_{k=1}^\infty\eta^k<\infty$ a.s.. Thus $\alpha^k$ satisfies Lemma~\ref{RobbSieg} and we have 
\begin{equation*}
    \Delta^k_{\hat{w}}\to\Delta_{\hat{w}} \quad\textrm{a.s.\ for some } \Delta_{\hat{w}}\in[-M_1,\infty).
\end{equation*}

Now since $(\Delta^k_{\hat{w}})_{k\in\mathbb{N}}$ converges a.s.\ and is a.s.~bounded, so is the right hand side of (\ref{Delta>0}), and since $y^k-y^{k-1}\to0$ a.s.\ we deduce that $(x^k,y^k)_{k\in\mathbb{N}}$ is a.s.\ bounded. Since $x^k$ is a.s.\ bounded and $y^k-y^{k-1}\to0$ a.s.\ it follows that
\begin{equation} \label{qto0}
    \langle QA(x^k-\hat{x}),y^{k}-y^{k-1} \rangle \to 0 \quad \mathrm{a.s.}
\end{equation}
From (\ref{yto0}) and (\ref{qto0}) we know 
some of the terms in 
%
%
 $\Delta^k_{\hat{w}}$ 
converge to zero a.s., namely 
$$\langle QA(x^k-x),y^{k-1}-y^k \rangle  + \|y^{k-1}-y^k\|^2 \to 0 \quad \mathrm{a.s.},$$
thus the sequence $(\|x^k-\hat{x}\|^2_{\tau^{-1}} + \|y^k-\hat{y}\|^2_{QS^{-1}})_{k\in\mathbb{N}}$ converges a.s.~to $\Delta_{\hat{w}}$.
Finally, notice that the norm $\vertiii{w}^2 := \|x\|^2_{\tau^{-1}} + \|y\|^2_{QS^{-1}} $ is equivalent to the product norm $\|\cdot\|^2$ in $X\times Y$, hence since the sequence $(\vertiii{w^k-\hat{w}})_{k\in\mathbb{N}}$ converges a.s., so does $(\|w^k-\hat{w}\|)_{k\in\mathbb{N}}$.
\end{proof}

\begin{proof}[Proof of Proposition \ref{tozero1}-\ref{iv}] We follow closely the proof of (\cite{combettesPesquet}, Proposition~2.3).
Let $(\mathbf{\Omega},\mathcal{F},\mathbb{P})$ denote the probability space corresponding to the random sequence $(w^k)_{k\in\mathbb{N}}$,
let $\mathbf{F}$ be the set of solutions to (\ref{saddle}) and let $\mathbf{G}(w^k)$ be the set of weak cluster points of the sequence $(w^k)_{k\in\mathbb{N}}$. 

By Proposition~\ref{tozero1}-\ref{iii}, the sequence $(\|w^k-w\|)_{k\in\mathbb{N}}$ converges a.s.~for every solution $w\in\mathbf{F}$. 
Using (\cite{combettesPesquet},~Proposition~2.3-iii)), there exists $\Omega\in\mathcal{F}$ such that $\mathbb{P}(\Omega)=1$ and the sequence $(\|w^k(\omega)-w\|)_{k\in\mathbb{N}}$ converges for every $w\in\mathbf{F}$ and $\omega\in\Omega$. 
This implies, since $\mathbf{F}$ is nonempty, that $(w^k(\omega))_{k\in\mathbb{N}}$ is bounded and thus $\mathbf{G}(w^k(\omega))$ is nonempty for all $\omega\in\Omega$. 

By assumption, there exists $\tilde{\Omega}\in\mathcal{F}$ such that $\mathbb{P}(\tilde{\Omega})=1$ and
$\textbf{G}(w^k(\omega))\subset\mathbf{F}$ for every $\omega\in\tilde{\Omega}$. Let $\omega\in\Omega\cap\tilde{\Omega}$, then  $(\|w^k(\omega)-w\|)_{k\in\mathbb{N}}$ converges for all $w\in\textbf{G}(w^k(\omega))\subset\mathbf{F}$. By (\cite{combettesPesquet}, Proposition~2.3-iv)), we have that $(w^k)_{k\in\mathbb{N}}$ converges weakly almost surely to an element of $\mathbf{G}(w^k)\subset\mathbf{F}$.
\end{proof}

\begin{proof}[Proof of Proposition \ref{cluster_saddle}-\ref{2i}]

By definition of the iterates in Algorithm \ref{alg_2}, we have
\begin{equation*}
    (T_{\mathbb{S}^k}(x^{k+1},y^{k}))_i = y_i^{k+1} \textrm{ for every } i\in\{1,...,n\}
\end{equation*}
and by induction it is easy to check that $z^k = A^*y^k$ for every $k\geq0$, hence
\begin{equation*} 
\begin{aligned}
    &\bar{z}^{k+1} = z^{k+1} + \sum_{i\in\mathbb{S}^k} \frac1{p_i}\delta_i 
    \\ &= z^k + \sum_{i\in\mathbb{S}^k} (1+\frac1{p_i})\delta_i  \\
    &= A^*y^{k} + \sum_{i\in\mathbb{S}^k} \frac{1+p_i}{p_i} A_{i}^*(y_{i}^{k+1}-y_{i}^{k}) \\
    &= A^*y^{k} + \sum_{i\in\mathbb{S}^k} \frac{1+p_i}{p_i} A_{i}^* ((T_{\mathbb{S}^k}(x^{k+1},y^{k}))_{i}-y_{i}^{k}) .
\end{aligned}
\end{equation*}
Thus $(T_{\mathbb{S}^k}(x^{k+1},y^{k}))_0 = \mathrm{prox}_{\tau g}(x^{k+1} - \tau \bar{z}^{k+1}) = x^{k+2}$, which proves (\ref{itTs}).
\end{proof}

\begin{proof}[{Proof of Proposition \ref{cluster_saddle}-\ref{2ii}}]

Let $w$ be a fixed point of $T_{\mathcal{S}}$ for every instance of $\mathbb{S}^k$. By (\ref{pi}), for every $i\in\{1,...,n\}$ there exists an instance $\mathcal{S}$ such that $i\in\mathcal{S}$ and
\begin{equation*}
    y_i = w_i= (T_{\mathcal{S}} w)_i = \textrm{prox}_{\sigma_i f_i^*}(y_i+\sigma_i A_ix).
\end{equation*}
Furthermore, for any $\mathcal{S}$,
\begin{equation*} 
\begin{aligned}
    &x = w_0 = (T_\mathcal{S} w)_0 \\
    &= \textrm{prox}_{\tau g} \Big( x-\tau A^*y - \sum_{i\in\mathcal{S}} \frac{1+p_i}{p_i} \tau A_i^*((T_\mathcal{S} w)_i-y_i) \Big)
    \\ &= \textrm{prox}_{\tau g}(x-\tau A^*y). 
\end{aligned}
\end{equation*} 
These conditions on $x$ and $y$ define a saddle point (\cite{bredieslorenz}, 6.4.2). The converse result is direct.
\end{proof}

\begin{proof}[Proof of Proposition~\ref{cluster_saddle}-\ref{2iii}]
Let $v^k= (x^{k+1},y^{k})$ so that, by Proposition~\ref{cluster_saddle}-\ref{2i}, we have $v^{k+1} = T_{\mathbb{S}^k}v^k$. From Proposition~\ref{tozero1}-\ref{i} we know
\begin{equation} \label{n-n+1}
    \|v^{k} - v^{k-1}\|^2 
    \to 0 .
\end{equation}
Taking the conditional expectation and denoting $p_{\mathcal{S}} := \mathbb{P}(\mathbb{S}^k=\mathcal{S})$ gives
\begin{equation*}
\begin{aligned}
\mathbb{E}^{k+1}\|v^{k+1} - v^k\|^2
&= \mathbb{E}^{k+1}\|T_{\mathbb{S}^k} v^{k} - v^k\|^2 \\
&= \sum_{\mathcal{S}} p_\mathcal{S}\|T_\mathcal{S} v^{k} - v^k\|^2 .
\end{aligned}
\end{equation*}
Thus by (\ref{n-n+1})
\begin{equation} \label{Tstozero}
   T_\mathcal{S} v^k-v^k \to 0 \quad\textrm{a.s.}
\end{equation}
for every $\mathcal{S}$ such that $p_\mathcal{S}>0$.
Now, let $w^{\ell_k}\rightharpoonup w^*$ be a weakly convergent subsequence. From~(\ref{n-n+1}) we know $ y^{k}-y^{k-1}\to0 $ a.s. and so $v^{\ell_k}$ also converges weakly to $w^*$.
This, together with~(\ref{Tstozero}) implies, for all $u\in X\times Y$ and all $\mathcal{S}$ such that $p_\mathcal{S}>0$,
\begin{equation*}
     \langle w^*, u\rangle = \lim_{k\to\infty} \langle v^{\ell_k}, u\rangle 
     = \lim_{k\to\infty} \langle T_{\mathcal{S}}v^{\ell_k}, u \rangle
     \; \textrm{a.s.}
\end{equation*}
and, by the weak sequential continuity of $T_\mathcal{S}$ (Assumption~\ref{assu}), it follows that
\begin{equation*}
    \lim_{k\to\infty} \langle T_{\mathcal{S}}v^{\ell_k}, u \rangle = \langle T_{\mathcal{S}}\big(\lim_{k\to\infty}v^{\ell_k}\big), u \rangle = \langle T_\mathcal{S}w^*, u\rangle 
\end{equation*}
almost surely. Hence $w^*$ is almost surely a fixed point of $T_{\mathcal{S}}$ for each instance of $\mathbb{S}^k$. By Proposition~\ref{cluster_saddle}-\ref{2ii}, $w^*$ is a saddle point. \end{proof}

\section{Step size parameters}
\label{sec:step}

Theorem~\ref{main} requires a choice of step size parameters $\tau, \sigma_i$ that satisfy condition (\ref{assu_i}), i.e.~$\|D\|<1$. 
In this section we illustrate how to choose adequate step sizes for specific examples of random samplings, starting with the general case of not necessarily strongly convex functions. Furthermore, if the functionals $g,f_i^*$ are strongly convex, then optimal step sizes can be determined that offer linear convergence rate~\cite{spdhg}.

\subsection{The general convex case}
\subsubsection{Serial sampling}
For serial sampling, a valid choice of step size parameters is given by
\begin{equation} \label{tausig_serial}
    \tau\sigma_i\|A_i\|^2 < p_i \quad\textrm{for every } i\in\{1,...,n\}.
\end{equation}
Indeed this implies $\|D\|<1$~(\ref{assu_i}) since, for serial sampling, (\ref{tausig_serial}) gives
\begin{equation*}
\begin{aligned}
\|D^{1/2}z\|^2 &= \langle D z,z \rangle 
= \mathbb{E}\|C_\mathbb{S}^*Qz\|^2 \\
&= \sum_{i=1}^n p_i \| C_i^*p_i^{-1}z_i \|^2 \\
&= \sum_{i=1}^n p_i^{-1} \tau\sigma_i \|A_i^*z_i \|^2
< \|z\|^2.
\end{aligned}
\end{equation*}
This insight can be generalized as follows.

\subsubsection{\textit{b}-serial sampling}
Consider a partition of the set $\{1,...,n\}$ into $m$ blocks $I_j$, i.e. 
\begin{equation*}
    \bigcup_{j=1}^m I_j = \{1,...,n\}, \quad 
I_{j}\cap I_l = \emptyset \text{ for all } j\neq l.
\end{equation*}
At every iteration we select a single block $j\in\{1,...,m\}$ and update every index $i\in I_j$. We also assume every index $i$ within a block $I_j$ to have the same dual step size~$\sigma_i$.
For such a sampling, step size condition (\ref{tausig_serial}) reads
\begin{equation} \label{assu_j}
    \tau \sigma_j \|\tilde{A}_j\|^2 < \tilde{p}_j \quad \textrm{ for } i\in I_j, \; j\in\{1,...,m\}  
\end{equation}
where $\tilde{p}_j$ is the probability of choosing block $I_j$ and $\tilde{A}_j:X\to\prod_{i\in I_j}Y_i$ is the operator 
\begin{equation*}
    \tilde{A}_jx = (A_ix)_{i\in I_j}.
\end{equation*}
We refer to this random process as \textit{$b$-serial sampling} when all subsets have size $b$, i.e.~$|I_j|=b$ for all $j\in\{1,...,m\}$, with $1\leq b \leq n$.

{In general, this form of sampling is equivalent to serial sampling where the operators $\tilde{A}_j$ have taken the role of the $A_i$, however, we make this distinction to focus on parameter $b$, the batch size, which will be an important variable for our numerical experiments. We will still refer to it as serial sampling when referring to the case $b=1$.}

\subsubsection{\textit{b}-nice sampling}
Consider now the random sampling $\mathbb{S}$ that randomly selects $b$ elements from $\{1,...,n\}$ at each iteration, i.e.~every instance of $\mathbb{S}$ is a random subset of size $b$. Clearly at each iteration there are $\binom{n}{b}$
possible choices.
This process is different from $b$-serial sampling, since the subsets are not disjoint and do not form a partition of $\{1,...,n\}$. 

In this paper we will assume {uniform} $b$-nice sampling, i.e.~all possible instances of $\mathbb{S}$ have equal probability. 
In this case, it is easy to see the probabilities $p_i = \mathbb{P}(i\in\mathbb{S})$ are given by $p_i = {b}/{n}$, and condition~(\ref{tausig}) reads
%
%
\begin{equation} \label{assu_b3}
  \frac{\tau\sigma n^2}{b^2}
  \|\mathbb{E}(A_{\mathbb{S}}A_{\mathbb{S}}^*)\| < 1. 
\end{equation}
%


\subsection{The strongly convex case}
Algorithm~\ref{alg_2} has linear convergence for strongly convex functionals $g,f_i^*$ \cite{spdhg}. 

\begin{theorem}
\label{Theo61}
    Let $g,f_i^*$ be strongly convex with parameters $\mu_g,\mu_i>0$ for $i\in\{1,...,n\}$ and let $(\hat{x},\hat{y})$ be the unique solution of~(\ref{saddle}). Let $D$ satisfy 
    \begin{equation} \label{vitheta}
        \|D\| < \frac{1}{\theta}
    \end{equation}
    where the extrapolation parameter $\theta\in(0,1)$ satisfies the lower bounds
    \begin{equation} \label{lbtheta}
    \begin{aligned}
        \theta \geq \frac{1}{1+2\mu_g\tau}, \quad \theta \geq \max_{i} ( 1-2\frac{\mu_i\sigma_ip_i}{1+2\mu_i\sigma_i} ).
    \end{aligned}
    \end{equation}
    Then Algorithm~\ref{alg_2} converges with rate $\mathcal{O}(\theta^k)$, i.e.\ the iterates $x^k,y^k$ satisfy
    \begin{equation} \label{theta^k}
    \begin{aligned}
    &\mathbb{E}( c\|x^k-\hat{x}\|^2_\mathbf{X} + \|y^k-\hat{y}\|^2_\mathbf{Y} )
    \\ & \leq \theta^k ( \|x^0-\hat{x}\|^2_\mathbf{X} + \|y^0-\hat{y}\|^2_\mathbf{Y} )
    \end{aligned}
    \end{equation}
    where $c= 1-\theta\|Q(\mathbb{E}( A_{\mathbb{S}}A_{\mathbb{S}}^*)Q\|$ and the operators $\mathbf{X},\mathbf{Y}$ are given by
    $\mathbf{X} = \tau^{-1}+2\mu_g$ and $\mathbf{Y} = {(S^{-1}+2M)Q}$ with $M=\mathrm{diag}(\mu_1,...,\mu_n)$.
\end{theorem}
\begin{proof}
This follows from~(\cite{spdhg}, Theorem~6.1), where we have replaced the ESO parameter condition $v_i<\theta^{-1}{p_i}$ with the equivalent condition $\|D\|<{\theta}^{-1}$, as explained in Lemma~\ref{BESO}.
\end{proof}

Furthermore, it is possible to estimate the optimal (smallest) value of $\theta$ by equating the lower bounds in (\ref{lbtheta}) together with step size condition (\ref{vitheta}) \cite{spdhg}. 
For instance, from (\ref{tausig}) we know that, for uniform step sizes, condition (\ref{vitheta}) is satisfied by 
\begin{equation*} 
    \tau\sigma\|B\| \theta =  \rho^2
\end{equation*}
where $B=Q\mathbb{E}(A_{\mathbb{S}}A_{\mathbb{S}}^*)Q$ and $\rho\in(0,1)$. This together with (\ref{lbtheta}) yields
\begin{equation} \label{theta_arb}
    \theta = 
    \max_i 1 - \frac{2p_i}{1+\sqrt{\beta_i}}
\end{equation}
where ${\beta_i} = 1+{\|{B}\|p_i}/{(\mu_g\mu_i \rho^2)}$, and the corresponding optimal step sizes are
\begin{equation} \label{tausig_arb}
\begin{aligned}
    \sigma &= \min_i\frac{\mu_i^{-1}}{ \sqrt{\beta_i} -1 }, \\
    \tau &= \min_i\frac{\mu_g^{-1}p_i}{1 - 2p_i + \sqrt{\beta_i}}.
\end{aligned}
\end{equation}

\subsubsection{Serial sampling}
Better theoretical convergence rates can be computed if more is known about the sampling. In particular, 
optimal convergence rates for uniform and non-uniform serial sampling have been proposed in~\cite{spdhg}. These are, for uniform serial sampling,
%
\begin{equation*} 
     \theta_{us} 
     = 1 - \frac{2}{n + n\max_i\sqrt{\alpha_i}}
\end{equation*}
where $\alpha_i = 1+{\|A_i\|^2}/({\mu_g\mu_i\rho^2})$, $p_i=b/n$ and step size parameters given by
\begin{equation} \label{step_unif}
\begin{aligned}
    \sigma_i &= \frac{\mu_i^{-1}}{\max_j\sqrt{\alpha_j}-1}, \\
    \tau &= \frac{\mu_g^{-1}}{n-2+n\max_j\sqrt{\alpha_j}},
\end{aligned}
\end{equation}
for $i\in\{1,...,n\}$, and, for serial sampling with optimized probabilities,  
\begin{equation*} 
    \theta_{os} 
    = 1 - \frac{2}{n + \sum_{i=1}^n\sqrt{\alpha_i}}
\end{equation*}
where the optimal probabilities are found to be 
\begin{equation} \label{opt_prob}
    p_i = \frac{1+\sqrt{\alpha_i}}{n+\sum_{j=1}^n\sqrt{\alpha_j}}, \quad i\in\{1,...,n\}
\end{equation}
and the step size parameters are
\begin{equation} \label{step_opt}
\begin{aligned}
    \sigma_i &= \frac{\mu_i^{-1}}{\sqrt{\alpha_i}-1}, \\
    \tau &= \frac{\mu_g^{-1}}{n-2+\sum_{j=1}^n\sqrt{\alpha_j}},
\end{aligned}
\end{equation}
for $i\in\{1,...,n\}$. 
In general $\theta_{os}\leq\theta_{us}$, as the former imposes less restrictions over probabilities $p_i$.

\subsubsection{\textit{b}-serial sampling}
These last results are also useful for $b$-serial sampling. To see this we define $\tilde{A}_j$ and $\tilde{p}_j$ as in~(\ref{assu_j}) and $\tilde{f}^*_j$ as
\begin{equation*}
    \tilde{f}^*_j(\tilde{y}_j) = \sum_{i\in I_j} f_i(y_i), \quad
\tilde{y}_j \in \prod_{i\in I_j} Y_i,
\quad j\in\{1,...,m\}. 
\end{equation*}
With this notation, our original (strongly convex) saddle point problem~(\ref{saddle}) 
becomes 
\begin{equation*} \label{saddle_b}
    \hat{x},\hat{y} \in \arg\min_{x\in X}\max_{\tilde{y}\in Y} \sum_{j=1}^m \langle \tilde{A}_j x,\tilde{y}_j \rangle - \tilde{f}_j^*(\tilde{y}_i) + g(x)
\end{equation*}
where $\tilde{y} = (\tilde{y}_1,...,\tilde{y}_m)$ and $\tilde{f}^*_j$ are strongly convex with parameters $\tilde{\mu}_j = \min\{\mu_i\,|\,i\in I_j\}$. Thus Lemma~\ref{Theo61} guarantees the linear convergence of Algorithm~\ref{alg_2} and, as before, the optimal convergence rates are 
\begin{equation} \label{theta_serial2}
\begin{aligned}
    \theta_{us} 
    &= 1 - \frac{2}{m+m\max_j\sqrt{\tilde{\alpha}_j}} \\
    \theta_{os} 
    &= 1 - \frac{2}{m + \sum_{j=1}^m\sqrt{\tilde{\alpha}_j}}
\end{aligned}
\end{equation}
with $\tilde{\alpha}_j = 1+{\|\tilde{A}_j\|^2}/{(\mu_g\tilde{\mu}_j\rho^2)}$. In both cases, the optimal step size parameters $\tau,\tilde{\sigma}_j$ are also given as in~(\ref{step_unif}) or~(\ref{step_opt}). The optimal probabilities are given by~(\ref{opt_prob}).

Notice how the optimal convergence rates $\theta_{us}$ and $\theta_{os}$ depend on $b$. In particular, choosing $b=n$ gives us the optimal step sizes for the PDHG algorithm~(\ref{alg_spdhg}), i.e.~SPDHG with full sampling. In this case, $m=1$ and the optimal convergence rate is
\begin{equation} \label{theta_full}
    \theta_{us} = \theta_{os} 
    = 1 - \frac{2}{1+\sqrt{ 1+\frac{\|A\|^2}{\mu_g\mu_{f^*}\rho^2} }}
\end{equation}
where $\mu_{f^*} = \min\{\mu_1,...,\mu_n\}$ is the convexity parameter of $f^*(y) = \sum_{i=1}^nf_i^*(y_i)$.

Notice as well that there is more than one way to partition the set $\{1,...,n\}$ into $m$ subsets $I_j$ of size $b$. For instance, if $n$ is divisible by $b$, the number $\mathbf{k}$ of different partitions of $\{1,...,n\}$ into subsets of size $b$ is 
\begin{equation} \label{partitions}
    \mathbf{k}(n,b) = \prod_{j=1}^{\frac{n}{b}}
    \left( \begin{tabular}{c} $jb-1$ \\ $b-1$ \end{tabular} \right) .
\end{equation} 
Moreover, the optimal convergence rates $\theta_{us}$ and $\theta_{os}$ will also depend on which partition we use, since the subsets $I_j$ define in turn the values $\tilde{\mu}_j$ and $\|\tilde{A}_j\|$. In Section~\ref{sec:MRI}, we will see examples of how using a different partition of $\{1,...,n\}$ can improve the convergence rate of SPDHG with $b$-serial sampling.

\subsubsection{\textit{b}-nice sampling}
For uniform $b$-nice sampling we use rate~(\ref{theta_arb}) with uniform probability $p_i=b/n$, and we denote
\begin{equation}\label{theta_bnice}
    \theta_{un} = 1 - \frac{2b}{n + n\max_i\sqrt{\beta_i}} 
\end{equation}
with step size parameters given as in~(\ref{tausig_arb}).

This quantity also depends strongly on the batch size $b$ since the probabilities $p_i$ and thus also the norm of the operator ${B}$ are determined by $b$.
In particular, choosing $b=n$ in (\ref{theta_bnice}) results in $p_i=1$ and $\|{B}\|=\|AA^*\| = \|A\|^2$, hence we recover the same full sampling convergence rate~(\ref{theta_full}) from $b$-serial sampling.

\section{Parallel MRI reconstruction}
\label{sec:MRI}

In this section we take real MRI data and perform parallel MRI reconstruction with sensitivity encoding using SPDHG~\cite{mri,sense}. For more examples on stochastic methods applied to parallel MRI reconstruction, see~\cite{ehrhardt2014joint,hager2015alternating,OngExtremeMRI,oscanoa2023coil}. For our experiments, the data have been taken from the NYU fastMRI dataset \cite{fastmri, fastmri_arxiv}.

\begin{figure*}[th]
\centering
    \begin{subfigure}[b]{0.34\textwidth}
    \includegraphics[width=0.99\textwidth]{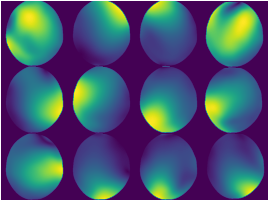}
    \vspace{-3.5mm}
    \caption{Coil sensitivities}
    \label{fig:subcoils}
    \end{subfigure}
    \hspace{1mm}
    \begin{subfigure}[b]{0.3\textwidth}
    \centering
    \includegraphics[width=0.9\textwidth]{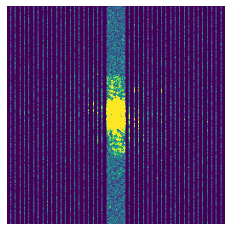}
    \caption{Data $\textcolor{black}{b_1}$ in $k$-space}
    \label{fig:coilsb}
    \end{subfigure}
    \begin{subfigure}[b]{0.3\textwidth}
    \centering
    \includegraphics[width=0.9\textwidth]{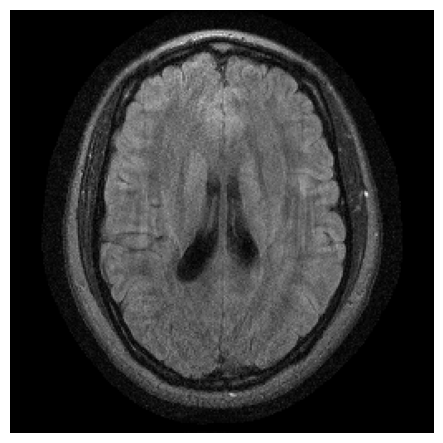}
    \caption{Reconstructed image $\textcolor{black}{\hat{x}}$ }
    \end{subfigure}
\caption{ \textbf{Parallel MRI reconstruction with $n=12$ coils.} (a) Spatial visualization of coil sensitivities for an MRI scanner with 12 electromagnetic coils distributed uniformly in a circle. (b) Data collected by a single coil, encoded in Fourier space. (c) A reconstructed image obtained using sigpy's SenseRecon method~\cite{sigpy}. }
\label{fig:coils}
\end{figure*}

In parallel MRI reconstruction, the goal is obtaining a set of images from a collection of raw data. The data is collected by an array of receiver coils within an MRI scanner. 
For a system with $n$ coils, we are given $n$ data $b_1,...,b_n$ which relate to the inverse problems
$$b_i=A_ix+\eta_i,$$
where each $A_i:\mathbb{C}^{d_1}\to \mathbb{C}^{d_2}$ is the forward operator from the signal space to the sample space, and $\eta_i\in \mathbb{C}^{d_2}$ represents random noise.
Figure~\ref{fig:coilsb} shows a visualisation of a data vector $b_i$ collected by a single coil. 

The individual sensitivity of each receiver coil may be regarded as a modification of the encoding operators, in a process known as sensitivity encoding~\cite{pruessmann1998coil}.
The forward operators are given by 
$$A_i = S\circ F\circ C_i,$$
where $S:\mathbb{C}^{d_1}\to\mathbb{C}^{d_2}$ is a subsampling operator, $F:\mathbb{C}^{d_1}\to\mathbb{C}^{d_1}$ represents the discrete Fourier transform and $C_ix = c_i\cdot x $ is the element-wise multiplication of $x$ and the $i$-th coil-sensitivity map $c_i\in \mathbb{C}^{d_1}$. 

Figure~\ref{fig:subcoils} shows the magnitude of the coil sensitivity maps $c_i$, obtained using sigpy's EspiritCalib method~\cite{sigpy} on a dataset collected using $n=12$ coils, arranged uniformly in a circle around a patient's head. 

We can reconstruct an image $\hat{x}$ from the data $b_i$ by solving the regularized least-squares problem
\begin{equation} \label{mri_model_g}
    \min_x \sum_{i=1}^n \frac12\|A_ix-b_i\|^2_2 + \lambda_1\|\nabla x\|_1 + \frac{\lambda_2}2\|x\|_2^2
\end{equation}
where $\lambda_1$ and $\lambda_2$ are regularization parameters, and we recover our convex minimization template~(\ref{min}) by identifying $\mathbb{C}^d$ with $\mathbb{R}^{2d}$, and setting $X=\mathbb{R}^{2d_1}$, $Y_i=\mathbb{R}^{2d_2}$, and  
$ f_i(y) = \frac12\|y - b_i\|^2 $,
$ g(x) = \lambda_1\|\nabla x\|_1 + \frac{\lambda_2}2\|x\|_2^2. $

In order to test the performance of SPDHG we investigate the convergence of its iterations $(x^k,y^k)$ at every epoch. We consider an epoch as the number of iterations required to perform roughly the same amount of computational work (e.g.~linear operations) as one iteration of the deterministic PDHG method. In particular, for $b$-serial and $b$-nice samplings, where one epoch is roughly equivalent to $m=n/b$ iterations, we define the \textit{relative primal error} (at epoch $k$)
as
\begin{equation} \label{e_b}
    \textbf{e}_b(k) := \frac{\|x^{mk}-\hat{x}\|}{\|\hat{x}\|},
\end{equation}
where $\hat{x}$ is a solution of~(\ref{mri_model_g}). In our experiments, $\hat{x}$ is obtained using 
the deterministic PDHG method~\cite{combettesPesquet} for $(10^5)$ iterations.

\begin{figure*}[t]
    \centering
    \includegraphics[width=0.9\textwidth]{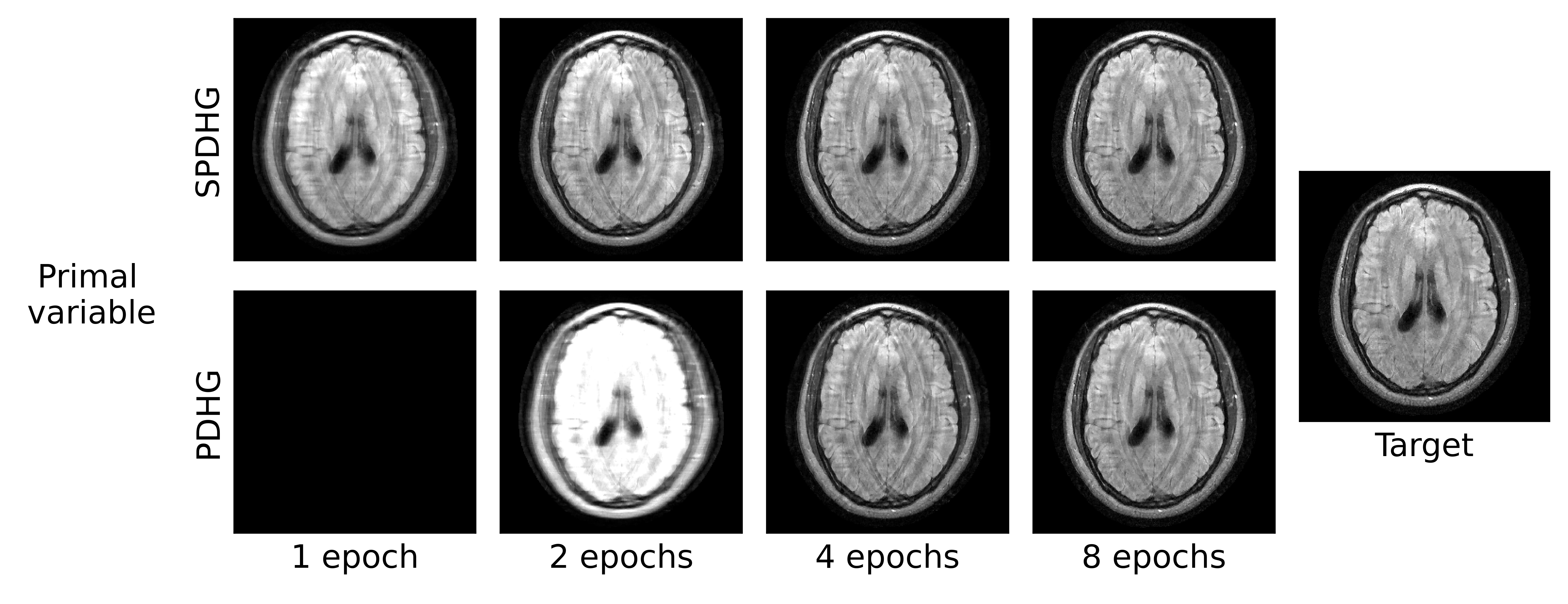}
    \includegraphics[width=0.9\textwidth]{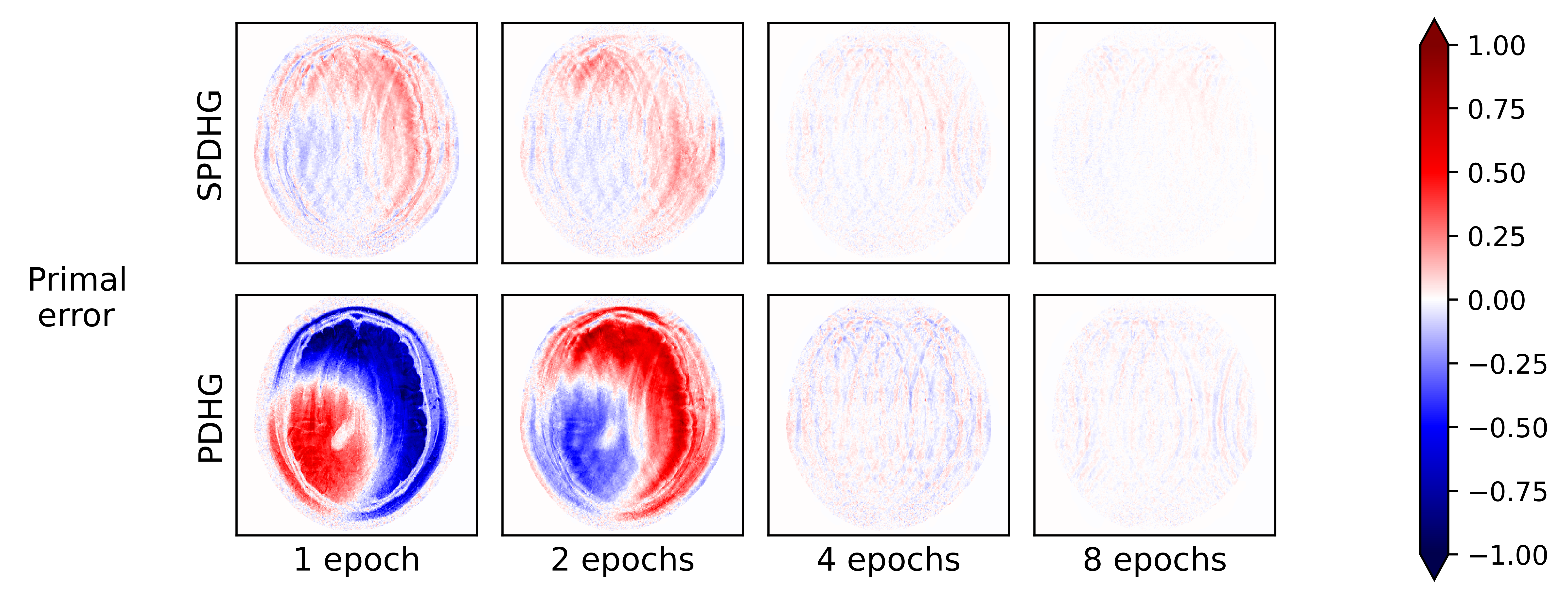}
\caption{\textbf{Visual comparison of PDHG and SPDHG.} Top: Visualisation of the primal variable $x^{mk}$. Bottom: Real part of the primal error $x^{mk}-\hat{x}$. 
While both methods converge to the same solution, the stochastic SPDHG progresses faster than the deterministic PDHG. }
\label{fig:brains}
\end{figure*}

Similarly, the convergence rate~$\theta$ from Lemma~\ref{Theo61} holds at every iteration~$k$, hence for $b$-serial and $b$-nice sampling we can define the \textit{convergence rates per epoch} as
\begin{equation} \label{theta_epoch}
    \mathbf{\vartheta}_{us} = \theta_{us}^m, \quad
    \mathbf{\vartheta}_{os} = \theta_{os}^m, \quad 
    \mathbf{\vartheta}_{un} = \theta_{un}^m
\end{equation}
with $\theta_{us}$, $\theta_{os}$ and $\theta_{un}$ defined as in (\ref{theta_serial2}) and (\ref{theta_bnice}).

All our experiments are implemented in Python using the Operator Discretization Library (ODL) \cite{odl}. The code for SPDHG's algorithm is based on the original implementation from~\cite{spdhg}. Operator norms are computed using the power method in ODL.
The proximity operator $\textrm{prox}_{\tau g}$ is approximated iteratively using the Fast Iterative Shrinkage/Thresholding Algorithm (FISTA) with warm starting~\cite{FISTA}. 
The number of iterations of FISTA per epoch of SPDHG is kept constant so different samplings can be compared.
The code files necessary to reproduce the figures in this paper are available at 
{https://github.com/Eric-Baruch/spdhg}. 

\subsection{SPDHG versus PDHG}

Figures~\ref{fig:brains} and~\ref{fig:b1} show the result of solving~(\ref{mri_model_g}) with $\lambda_1=10^2$ and $\lambda_2=10^{-2}$ using SPDHG with optimal $b$-serial sampling for two different values of $b$. These are $b=1$, i.e.\ serial sampling, and $b=12$, i.e.\ the deterministic PDHG. In this case, 1 epoch corresponds to $n=12$ iterations of SPDHG or 1 iteration of PDHG. Figure~\ref{fig:brains} shows the primal iterates of SPDHG approximate the solution faster than those of the deterministic PDHG.

\begin{figure}[!htb]
\includegraphics[width=0.48\textwidth]{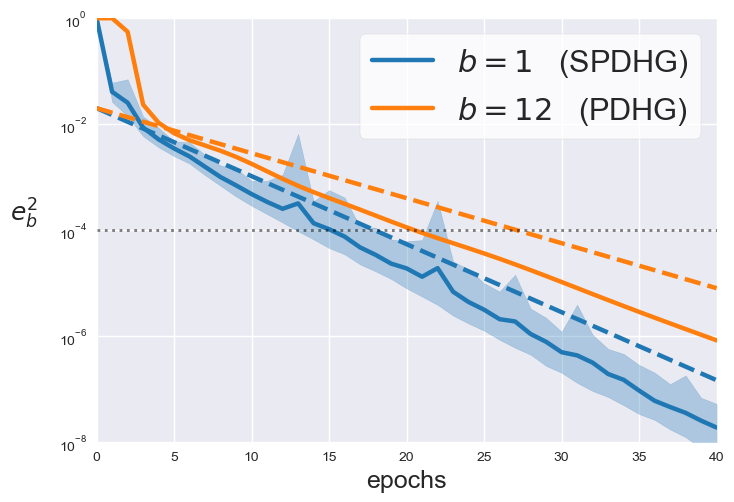}
\caption{\textbf{Quantitative comparison of SPDHG and PDHG.} Relative primal error squared $e_b^2$ (\ref{e_b}) for Algorithm~\ref{alg_2} with optimal $b$-serial sampling, for batch size 1 and $n$.
The dashed lines show the theoretical convergence rates $\vartheta_{os}$~(\ref{theta_epoch}). The solid blue curve is the average of 40 independent runs contained within the shaded region. 
The dotted grey line indicates the iterations have visually converged to the solution. SPDHG reaches this mark significantly faster than PDHG, both theoretically and empirically.
}
\label{fig:b1}
\end{figure}

\begin{figure}[!htb]
    \includegraphics[width=0.48\textwidth]{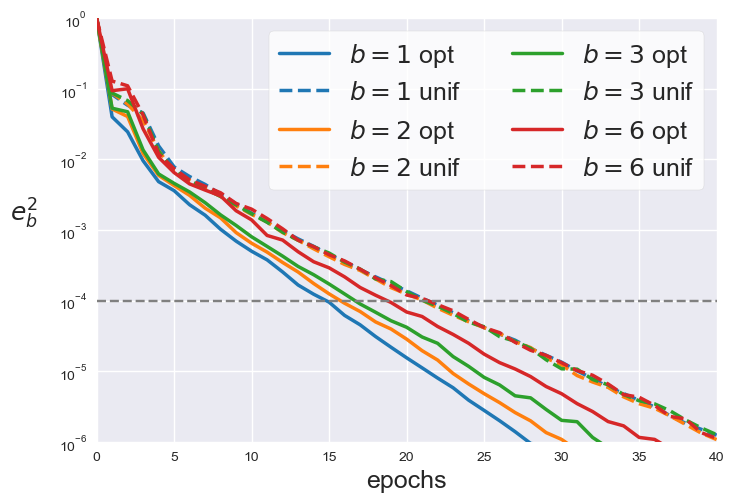}
    \caption{\textbf{Optimal sampling versus uniform sampling.} Performance of SPDHG for different values of $b$, using step size parameters with uniform probabilities~(\ref{step_unif}) and with optimal probabilities~(\ref{step_opt}). Each curve is the average of 40 independent runs. Convergence speed is significantly increased by sampling with optimal probabilities. }
\label{fig:bserial_bfair}
\end{figure}

Figure~\ref{fig:b1} shows the relative primal error $\mathbf{e}_b$ (solid) along with its theoretical convergence rate per epoch $\vartheta_{os}$ (dashed). The theoretical rate per epoch for the SPDHG curve is computed to be $\vartheta_{os} \approx 0.7439$ while the rate for PDHG is $\vartheta_{os} \approx 0.8222$.
To minimise randomness, the blue curve is the average of 40 independent runs, and the shaded region around it illustrates the minimum and maximum values attained for all runs. The orange curve is deterministic and has no variance. We see that SPDHG has faster convergence than its deterministic counterpart, both theoretically and empirically, and that the theoretical rate is indicative of its performance.

\subsection{Optimal versus uniform sampling}

In~(\ref{mri_model_g}), functionals $f^*_i,g$ are strongly convex with convexity parameters $\mu_i = 1$ for all $i$ and $\mu_g = \lambda_2$, respectively. Hence as in Section~\ref{sec:step} we can use Lemma~\ref{Theo61} to determine, for each type of sampling, the optimal step size parameters $\tau, \sigma_i$ such that the convergence rate is minimised.
We compute these optimal step size parameters for two types of sampling: optimal $b$-serial sampling and uniform $b$-serial sampling. 

Figure~\ref{fig:bserial_bfair} shows the result of using these parameters for different values of $b$. We see that for all $b$ the convergence is  faster for optimal sampling compared to uniform sampling, i.e.\ the convergence speed is significantly improved by adjusting the probabilities and the step size parameters.

\subsection{Best partition for optimal sampling}

{When considering $b$-serial sampling, it is necessary to partition the data into $n/b$ subsets of size $b$. A natural partition of $\{1,...,n\}$ is the \textit{consecutive} partition, consisting of the subsets 
$$ \quad I_j=\{(j-1)b+1,\,(j-1)b+2,\,...\,,\,(j-1)b+b\}. $$ 
%
%
Of course
this is not the only way to partition $\{1,...,n\}$ into $m$ subsets of size $b$. For instance, from~(\ref{partitions}) we know the number of partitions for $n=12$ and $b=6$ is $ 
462 $, 
while for $b=3$ the number is 
$
15\,400 $.
For each one of these partitions it is possible to compute the optimal convergence rates per epoch $\vartheta_{os}$ and $\vartheta_{us}$.

\begin{figure}[ht]
    \centering
    \includegraphics[width=0.47\textwidth]{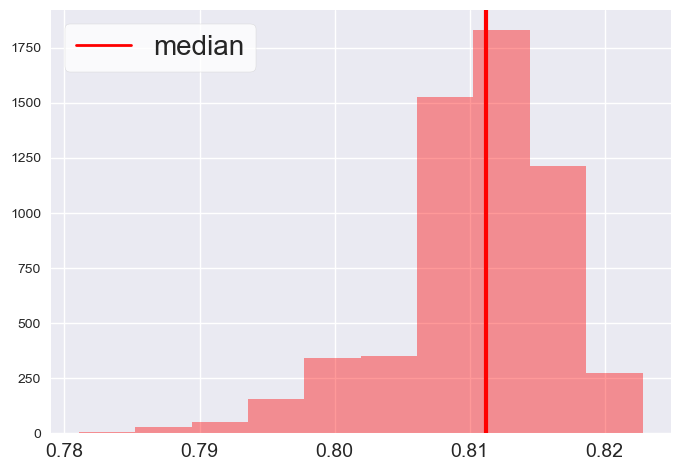}
    \caption{ \textbf{
    Distribution of $\vartheta_{os}$ for $b=4$. } Theoretical convergence rates $\vartheta_{os}$ for optimal $b$-serial sampling, computed as in~(\ref{theta_epoch}) for all 5775 possible partitions for $n=12$ and $b=4$. The choice of partition has a strong effect on the guaranteed performance of SPDHG. }
    \label{fig:hista}
    \end{figure}

\begin{figure*}[t]
    \centering
    \includegraphics[width=0.98\textwidth]{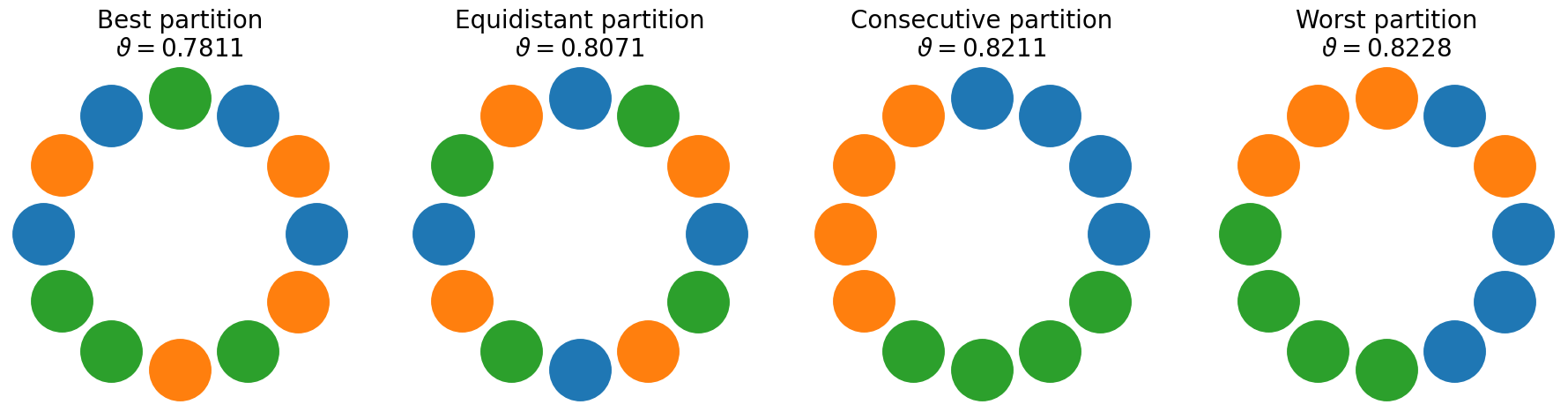}
    %
\caption{ \textbf{Different partitions for $b=4$ and their convergence rates for optimal $b$-serial sampling.} The best and the worst of all possible $5775$ partitions of batch size $b=4$ according to their theoretical convergence rate $\vartheta_{os}$ (\ref{theta_epoch}), as well as two other partitions of interest, the equidistant and the consecutive partition.  }
\label{fig:besta}
\end{figure*}

\begin{figure}
\includegraphics[width=0.48\textwidth]{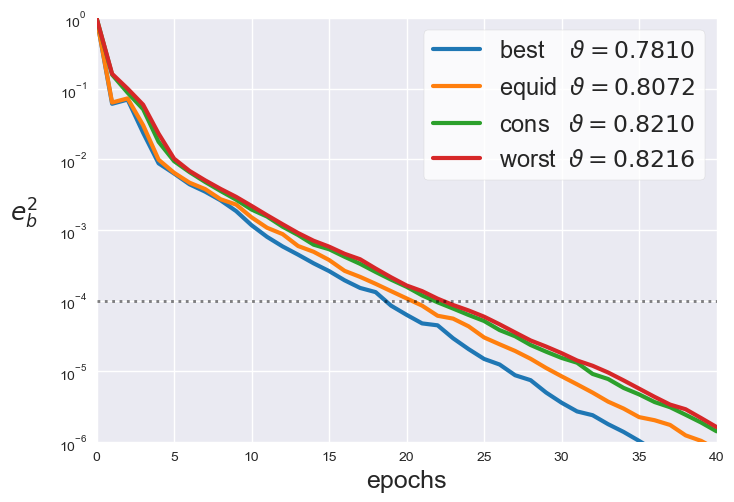}
\caption{ \textbf{Performance of different partitions for $b$=4}.
Relative primal error squared $\mathbf{e}^2_b$ for SPDHG with optimal $b$-serial sampling, using the partitions shown in Figure~\ref{fig:besta}. Every curve is the average of 40 independent runs. Performance is significantly improved by choosing the correct partition. The consecutive partition performs poorly although it is the most intuitive. }
\label{fig:bestb}
\end{figure}

Figure~\ref{fig:hista} 
shows the distribution of these rates over all possible 5775 partitions for $n=12$ and $b=4$ for optimal $b$-serial sampling. We see that some partitions result in significantly faster convergence rate than others. This naturally raises the question of what constitutes a good choice of partition.

%

From all the values in Figure~\ref{fig:hista}, 
we can identify the partitions that correspond to the best and worst convergence rate $\vartheta_{os}$. 
Figure~\ref{fig:besta} illustrates these partitions by showing their corresponding coils and their positions in the MRI scanner. 

Figure~\ref{fig:besta} also suggests that the physical location of the coils contributes to the quality of a partition. This is also supported by definitions~(\ref{theta_serial2}) and~(\ref{theta_epoch}), where the choice of partition defines the operators $\tilde{A}_j$.
%
%
%
With this intuition, Figure~\ref{fig:besta} includes two other partitions of interest: the \textit{consecutive} partition and the \textit{equidistant} partition. 

\begin{figure*}[ht]
    \centering
    \begin{subfigure}[b]{0.47\textwidth}
        \centering
        \includegraphics[width=\textwidth]{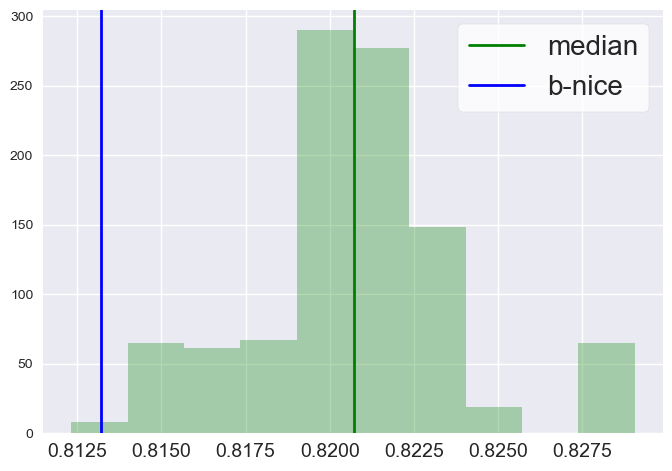}
        \vspace{1.2mm}
        \caption{Convergence rates $\vartheta_{us}$ for $b=4$}
    \label{fig:boxa}
    \end{subfigure}
    \hfill
    \begin{subfigure}[b]{0.47\textwidth}
        \centering
        \includegraphics[width=\textwidth]{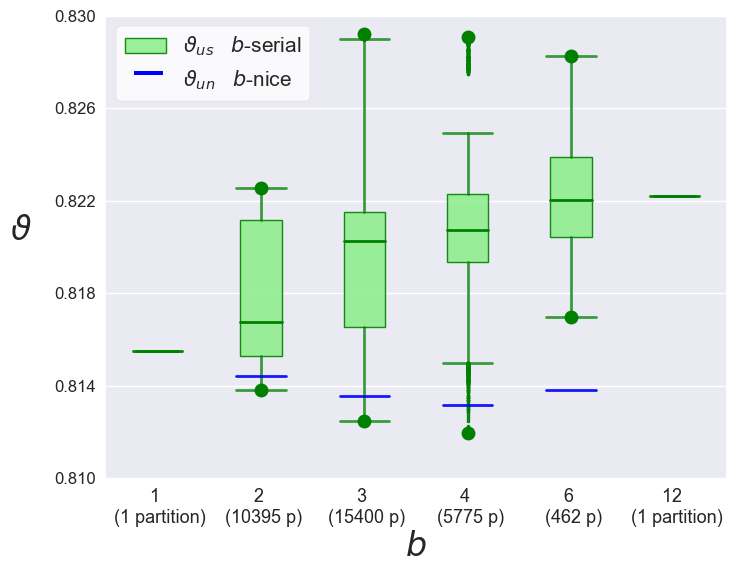}
        \caption{Convergence rates $\vartheta_{us}$ for all values of $b$}
    \label{fig:boxb}
    \end{subfigure}
    \caption{\textbf{Convergence rates for uniform $b$-serial sampling and uniform $b$-nice sampling.} Left: Convergence rates per epoch $\vartheta_{us}$ for uniform $b$-serial sampling for all 5775 possible partitions of batch size $b=4$. The convergence rate $\vartheta_{un}$ for uniform $b$-nice sampling does not depend on any partition. Right: Distribution of $\vartheta_{us}$
    for each $b$. Trivially, $b$-serial and $b$-nice are the same sampling for $b=1$, as well as for $b=n$. In every other case, $b$-nice sampling performs significantly better than the average partition of $b$-serial sampling, and almost as good as the best partition, if not better. }
    \label{fig:box}
\end{figure*}

Figure~\ref{fig:bestb} shows the performance of 
SPDHG using these four partitions. Notice how there is a significant improvement in performance by choosing the correct partition. {Furthermore, the performance for optimal $b$-serial sampling with the consecutive partition is almost as bad as using the worst possible partition.}
%

The histogram in Figure~\ref{fig:hista}
can also be computed for other values of $b$ in order to find the best partition for each $b$, i.e.\ the partition of batch size $b$ that yields the best convergence rate when using optimal $b$-serial sampling. Furthermore, we can also compute these histograms for uniform $b$-serial sampling for each $b$. In fact, each of the curves shown in Figure~\ref{fig:bserial_bfair} uses its corresponding best partition.

\subsection{\textit{b}-nice sampling versus \textit{b}-serial sampling }

Unlike $b$-serial sampling, $b$-nice sampling does not require data to be divided into subsets, thus its convergence $\vartheta_{un}$ its independent from any partition. Figure~\ref{fig:box} compares the performance of $b$-nice sampling against uniform $b$-serial sampling and its many possible partitions.

Figure~\ref{fig:boxa} shows the convergence rates $\vartheta_{us}$ for uniform $b$-serial sampling. 
The convergence rate $\vartheta_{un}$ for uniform $b$-nice sampling is shown in blue. 
We see that $b$-nice sampling has a faster convergence rate than $b$-serial sampling for most partitions. 

Figure~\ref{fig:boxb} shows the distribution of $\vartheta_{us}$ for all possible partitions for several values of $b$. As before, the rate $\vartheta_{un}$ for uniform $b$-nice sampling is independent of any partition. Trivially, for the special cases of $b=1$ and $b=n$ there is only one possible partition, and also the method defined by $b$-nice and $b$-serial sampling are the same, as they both correspond to either serial sampling for $b=1$ and full sampling for $b=n$.

Except from $b=6$, $b$-nice sampling performs better than $b$-serial sampling for most partitions. 
Notice that while the average partition of $b$-serial sampling performs better than full sampling (PDHG), the worst case the performance is significantly worse. 
This means that, when using uniform probabilities, if there is no information available about the partitions then $b$-nice sampling can be a good choice, since it guarantees good performance in any case, while $b$-serial requires a good choice of partition.

\section{Conclusions and Discussion}
We have closed a gap in the convergence analysis of SPDHG by extending its convergence guarantees to separable Hilbert spaces and any arbitrary sampling. We give a concrete strategy to find parameters that satisfy the required step size condition for all possible random samplings. For several specific random saplings, we determine theoretically optimal step size parameters.

We put these samplings to the test on parallel MRI reconstruction. 
Our experiments show that SPDHG performs better than the deterministic PDHG. Furthermore, sampling with optimised probabilities significantly improves convergence speed at no extra cost.

In addition, we show that the choice of random sampling also strongly influences the performance. In particular, when using $b$-serial sampling, physical properties of the MRI process such as the location of the coils can contribute to determining a good choice of partition. 

We conclude that SPDHG is an efficient stochastic method for relevant convex problems such as parallel MRI, and that the tools provided here help optimise its performance for different data-sampling strategies.

\section*{Acknowledgements}
MJE and CD acknowledge support from the EPSRC (EP/S026045/1). MJE is also supported by EPSRC (EP/T026693/1), the Faraday Institution
(EP/T007745/1) and the Leverhulme Trust (ECF-2019-478). EBG acknowledges the Mexican Council of Science and Technology (CONACyT).


\bibliography{sn-bibliography}


\begin{thebibliography}{42}
\ifx \bisbn   \undefined \def \bisbn  #1{ISBN #1}\fi
\ifx \binits  \undefined \def \binits#1{#1}\fi
\ifx \bauthor  \undefined \def \bauthor#1{#1}\fi
\ifx \batitle  \undefined \def \batitle#1{#1}\fi
\ifx \bjtitle  \undefined \def \bjtitle#1{#1}\fi
\ifx \bvolume  \undefined \def \bvolume#1{\textbf{#1}}\fi
\ifx \byear  \undefined \def \byear#1{#1}\fi
\ifx \bissue  \undefined \def \bissue#1{#1}\fi
\ifx \bfpage  \undefined \def \bfpage#1{#1}\fi
\ifx \blpage  \undefined \def \blpage #1{#1}\fi
\ifx \burl  \undefined \def \burl#1{\textsf{#1}}\fi
\ifx \doiurl  \undefined \def \doiurl#1{\url{https://doi.org/#1}}\fi
\ifx \betal  \undefined \def \betal{\textit{et al.}}\fi
\ifx \binstitute  \undefined \def \binstitute#1{#1}\fi
\ifx \binstitutionaled  \undefined \def \binstitutionaled#1{#1}\fi
\ifx \bctitle  \undefined \def \bctitle#1{#1}\fi
\ifx \beditor  \undefined \def \beditor#1{#1}\fi
\ifx \bpublisher  \undefined \def \bpublisher#1{#1}\fi
\ifx \bbtitle  \undefined \def \bbtitle#1{#1}\fi
\ifx \bedition  \undefined \def \bedition#1{#1}\fi
\ifx \bseriesno  \undefined \def \bseriesno#1{#1}\fi
\ifx \blocation  \undefined \def \blocation#1{#1}\fi
\ifx \bsertitle  \undefined \def \bsertitle#1{#1}\fi
\ifx \bsnm \undefined \def \bsnm#1{#1}\fi
\ifx \bsuffix \undefined \def \bsuffix#1{#1}\fi
\ifx \bparticle \undefined \def \bparticle#1{#1}\fi
\ifx \barticle \undefined \def \barticle#1{#1}\fi
\bibcommenthead
\ifx \bconfdate \undefined \def \bconfdate #1{#1}\fi
\ifx \botherref \undefined \def \botherref #1{#1}\fi
\ifx \url \undefined \def \url#1{\textsf{#1}}\fi
\ifx \bchapter \undefined \def \bchapter#1{#1}\fi
\ifx \bbook \undefined \def \bbook#1{#1}\fi
\ifx \bcomment \undefined \def \bcomment#1{#1}\fi
\ifx \oauthor \undefined \def \oauthor#1{#1}\fi
\ifx \citeauthoryear \undefined \def \citeauthoryear#1{#1}\fi
\ifx \endbibitem  \undefined \def \endbibitem {}\fi
\ifx \bconflocation  \undefined \def \bconflocation#1{#1}\fi
\ifx \arxivurl  \undefined \def \arxivurl#1{\textsf{#1}}\fi
\csname PreBibitemsHook\endcsname

\bibitem[\protect\citeauthoryear{Benning and Burger}{2018}]{benning2018modern}
\begin{barticle}
\bauthor{\bsnm{Benning}, \binits{M.}},
\bauthor{\bsnm{Burger}, \binits{M.}}:
\batitle{Modern regularization methods for inverse problems}.
\bjtitle{Acta numerica}
\bvolume{27},
\bfpage{1}--\blpage{111}
(\byear{2018})
\end{barticle}
\endbibitem

\bibitem[\protect\citeauthoryear{Boser et~al.}{1992}]{boserGuyonVapnik}
\begin{bchapter}
\bauthor{\bsnm{Boser}, \binits{B.E.}},
\bauthor{\bsnm{Guyon}, \binits{I.M.}},
\bauthor{\bsnm{Vapnik}, \binits{V.N.}}:
\bctitle{A training algorithm for optimal margin classifiers}.
In: \bbtitle{Proceedings of the Fifth Annual Workshop on Computational Learning Theory},
pp. \bfpage{144}--\blpage{152}
(\byear{1992})
\end{bchapter}
\endbibitem

\bibitem[\protect\citeauthoryear{Cevher et~al.}{2014}]{cevherBigData}
\begin{barticle}
\bauthor{\bsnm{Cevher}, \binits{V.}},
\bauthor{\bsnm{Becker}, \binits{S.}},
\bauthor{\bsnm{Schmidt}, \binits{M.}}:
\batitle{Convex optimization for big data: Scalable, randomized, and parallel algorithms for big data analytics}.
\bjtitle{IEEE Signal Processing Magazine}
\bvolume{31}(\bissue{5}),
\bfpage{32}--\blpage{43}
(\byear{2014})
\end{barticle}
\endbibitem

\bibitem[\protect\citeauthoryear{Shalev-Shwartz and Zhang}{2013}]{shalevZhang}
\begin{barticle}
\bauthor{\bsnm{Shalev-Shwartz}, \binits{S.}},
\bauthor{\bsnm{Zhang}, \binits{T.}}:
\batitle{Stochastic dual coordinate ascent methods for regularized loss minimization}.
\bjtitle{Journal of Machine Learning Research}
\bvolume{14}(\bissue{Feb}),
\bfpage{567}--\blpage{599}
(\byear{2013})
\end{barticle}
\endbibitem

\bibitem[\protect\citeauthoryear{Zhang and Xiao}{2017}]{zhangXiao}
\begin{barticle}
\bauthor{\bsnm{Zhang}, \binits{Y.}},
\bauthor{\bsnm{Xiao}, \binits{L.}}:
\batitle{Stochastic primal-dual coordinate method for regularized empirical risk minimization}.
\bjtitle{The Journal of Machine Learning Research}
\bvolume{18}(\bissue{1}),
\bfpage{2939}--\blpage{2980}
(\byear{2017})
\end{barticle}
\endbibitem

\bibitem[\protect\citeauthoryear{Fercoq et~al.}{2019}]{fercoq_etal2019}
\begin{bchapter}
\bauthor{\bsnm{Fercoq}, \binits{O.}},
\bauthor{\bsnm{Alacaoglu}, \binits{A.}},
\bauthor{\bsnm{Necoara}, \binits{I.}},
\bauthor{\bsnm{Cevher}, \binits{V.}}:
\bctitle{Almost surely constrained convex optimization}.
In: \bbtitle{International Conference on Machine Learning},
pp. \bfpage{1910}--\blpage{1919}
(\byear{2019}).
\bcomment{PMLR}
\end{bchapter}
\endbibitem

\bibitem[\protect\citeauthoryear{Patrascu and Necoara}{2017}]{patrascuNecoara2017}
\begin{barticle}
\bauthor{\bsnm{Patrascu}, \binits{A.}},
\bauthor{\bsnm{Necoara}, \binits{I.}}:
\batitle{Nonasymptotic convergence of stochastic proximal point methods for constrained convex optimization}.
\bjtitle{The Journal of Machine Learning Research}
\bvolume{18}(\bissue{1}),
\bfpage{7204}--\blpage{7245}
(\byear{2017})
\end{barticle}
\endbibitem

\bibitem[\protect\citeauthoryear{Rudin et~al.}{1992}]{ROF}
\begin{barticle}
\bauthor{\bsnm{Rudin}, \binits{L.I.}},
\bauthor{\bsnm{Osher}, \binits{S.}},
\bauthor{\bsnm{Fatemi}, \binits{E.}}:
\batitle{Nonlinear total variation based noise removal algorithms}.
\bjtitle{Physica D: nonlinear phenomena}
\bvolume{60}(\bissue{1-4}),
\bfpage{259}--\blpage{268}
(\byear{1992})
\end{barticle}
\endbibitem

\bibitem[\protect\citeauthoryear{Chambolle and Pock}{2016}]{chambolle2016introduction}
\begin{barticle}
\bauthor{\bsnm{Chambolle}, \binits{A.}},
\bauthor{\bsnm{Pock}, \binits{T.}}:
\batitle{An introduction to continuous optimization for imaging}.
\bjtitle{Acta Numerica}
\bvolume{25},
\bfpage{161}--\blpage{319}
(\byear{2016})
\end{barticle}
\endbibitem

\bibitem[\protect\citeauthoryear{Ehrhardt et~al.}{2019}]{ehrhardt2019faster}
\begin{barticle}
\bauthor{\bsnm{Ehrhardt}, \binits{M.J.}},
\bauthor{\bsnm{Markiewicz}, \binits{P.}},
\bauthor{\bsnm{Sch{\"o}nlieb}, \binits{C.-B.}}:
\batitle{Faster {PET} reconstruction with non-smooth priors by randomization and preconditioning}.
\bjtitle{Physics in Medicine \& Biology}
\bvolume{64}(\bissue{22}),
\bfpage{225019}
(\byear{2019})
\end{barticle}
\endbibitem

\bibitem[\protect\citeauthoryear{Fessler}{2020}]{mri}
\begin{barticle}
\bauthor{\bsnm{Fessler}, \binits{J.A.}}:
\batitle{Optimization methods for magnetic resonance image reconstruction: Key models and optimization algorithms}.
\bjtitle{IEEE Signal Processing Magazine}
\bvolume{37}(\bissue{1}),
\bfpage{33}--\blpage{40}
(\byear{2020})
\end{barticle}
\endbibitem

\bibitem[\protect\citeauthoryear{Hager et~al.}{2015}]{hager2015alternating}
\begin{barticle}
\bauthor{\bsnm{Hager}, \binits{W.}},
\bauthor{\bsnm{Ngo}, \binits{C.}},
\bauthor{\bsnm{Yashtini}, \binits{M.}},
\bauthor{\bsnm{Zhang}, \binits{H.-C.}}:
\batitle{An alternating direction approximate newton algorithm for ill-conditioned inverse problems with application to parallel mri}.
\bjtitle{Journal of the Operations Research Society of China}
\bvolume{3}(\bissue{2}),
\bfpage{139}--\blpage{162}
(\byear{2015})
\end{barticle}
\endbibitem

\bibitem[\protect\citeauthoryear{Pruessmann}{2006}]{pruessmann2006encoding}
\begin{barticle}
\bauthor{\bsnm{Pruessmann}, \binits{K.P.}}:
\batitle{Encoding and reconstruction in parallel mri}.
\bjtitle{NMR in Biomedicine: An International Journal Devoted to the Development and Application of Magnetic Resonance In vivo}
\bvolume{19}(\bissue{3}),
\bfpage{288}--\blpage{299}
(\byear{2006})
\end{barticle}
\endbibitem

\bibitem[\protect\citeauthoryear{Chambolle and Pock}{2011}]{chambollePock}
\begin{barticle}
\bauthor{\bsnm{Chambolle}, \binits{A.}},
\bauthor{\bsnm{Pock}, \binits{T.}}:
\batitle{A first-order primal-dual algorithm for convex problems with applications to imaging}.
\bjtitle{Journal of mathematical imaging and vision}
\bvolume{40}(\bissue{1}),
\bfpage{120}--\blpage{145}
(\byear{2011})
\end{barticle}
\endbibitem

\bibitem[\protect\citeauthoryear{Esser et~al.}{2010}]{esserZhangChan}
\begin{barticle}
\bauthor{\bsnm{Esser}, \binits{E.}},
\bauthor{\bsnm{Zhang}, \binits{X.}},
\bauthor{\bsnm{Chan}, \binits{T.F.}}:
\batitle{A general framework for a class of first order primal-dual algorithms for convex optimization in imaging science}.
\bjtitle{SIAM Journal on Imaging Sciences}
\bvolume{3}(\bissue{4}),
\bfpage{1015}--\blpage{1046}
(\byear{2010})
\end{barticle}
\endbibitem

\bibitem[\protect\citeauthoryear{{Pock} et~al.}{2009}]{CremersChambollePock}
\begin{bchapter}
\bauthor{\bsnm{{Pock}}, \binits{T.}},
\bauthor{\bsnm{{Cremers}}, \binits{D.}},
\bauthor{\bsnm{{Bischof}}, \binits{H.}},
\bauthor{\bsnm{{Chambolle}}, \binits{A.}}:
\bctitle{A algorithm for minimizing the {Mumford-Shah} functional}.
In: \bbtitle{2009 IEEE 12th International Conference on Computer Vision},
pp. \bfpage{1133}--\blpage{1140}
(\byear{2009})
\end{bchapter}
\endbibitem

\bibitem[\protect\citeauthoryear{Chambolle et~al.}{2018}]{spdhg}
\begin{barticle}
\bauthor{\bsnm{Chambolle}, \binits{A.}},
\bauthor{\bsnm{Ehrhardt}, \binits{M.J.}},
\bauthor{\bsnm{Richt{\'a}rik}, \binits{P.}},
\bauthor{\bsnm{Schönlieb}, \binits{C.-B.}}:
\batitle{Stochastic primal-dual hybrid gradient algorithm with arbitrary sampling and imaging applications}.
\bjtitle{SIAM Journal on Optimization}
\bvolume{28}(\bissue{4}),
\bfpage{2783}--\blpage{2808}
(\byear{2018})
\end{barticle}
\endbibitem

\bibitem[\protect\citeauthoryear{Schramm and Holler}{2022}]{Schramm_Holler}
\begin{botherref}
\oauthor{\bsnm{Schramm}, \binits{G.}},
\oauthor{\bsnm{Holler}, \binits{M.}}:
Fast and memory-efficient reconstruction of sparse poisson data in listmode with non-smooth priors with application to time-of-flight pet.
Physics in Medicine \& Biology
(2022)
\end{botherref}
\endbibitem

\bibitem[\protect\citeauthoryear{Fercoq and Bianchi}{2019}]{fercoqBianchi2019}
\begin{barticle}
\bauthor{\bsnm{Fercoq}, \binits{O.}},
\bauthor{\bsnm{Bianchi}, \binits{P.}}:
\batitle{A coordinate-descent primal-dual algorithm with large step size and possibly nonseparable functions}.
\bjtitle{SIAM Journal on Optimization}
\bvolume{29}(\bissue{1}),
\bfpage{100}--\blpage{134}
(\byear{2019})
\end{barticle}
\endbibitem

\bibitem[\protect\citeauthoryear{Gao et~al.}{2019}]{gao2019randomPD}
\begin{barticle}
\bauthor{\bsnm{Gao}, \binits{X.}},
\bauthor{\bsnm{Xu}, \binits{Y.-Y.}},
\bauthor{\bsnm{Zhang}, \binits{S.-Z.}}:
\batitle{Randomized primal--dual proximal block coordinate updates}.
\bjtitle{Journal of the Operations Research Society of China}
\bvolume{7}(\bissue{2}),
\bfpage{205}--\blpage{250}
(\byear{2019})
\end{barticle}
\endbibitem

\bibitem[\protect\citeauthoryear{Latafat et~al.}{2019}]{latafat2019randomPD}
\begin{barticle}
\bauthor{\bsnm{Latafat}, \binits{P.}},
\bauthor{\bsnm{Freris}, \binits{N.M.}},
\bauthor{\bsnm{Patrinos}, \binits{P.}}:
\batitle{A new randomized block-coordinate primal-dual proximal algorithm for distributed optimization}.
\bjtitle{IEEE Transactions on Automatic Control}
\bvolume{64}(\bissue{10}),
\bfpage{4050}--\blpage{4065}
(\byear{2019})
\end{barticle}
\endbibitem

\bibitem[\protect\citeauthoryear{Alacaoglu et~al.}{2022}]{alacaoglu}
\begin{barticle}
\bauthor{\bsnm{Alacaoglu}, \binits{A.}},
\bauthor{\bsnm{Fercoq}, \binits{O.}},
\bauthor{\bsnm{Cevher}, \binits{V.}}:
\batitle{On the convergence of stochastic primal-dual hybrid gradient}.
\bjtitle{SIAM Journal on Optimization}
\bvolume{32}(\bissue{2}),
\bfpage{1288}--\blpage{1318}
(\byear{2022})
\end{barticle}
\endbibitem

\bibitem[\protect\citeauthoryear{Guti{\'e}rrez et~al.}{2021}]{gutierrez}
\begin{bchapter}
\bauthor{\bsnm{Guti{\'e}rrez}, \binits{E.B.}},
\bauthor{\bsnm{Delplancke}, \binits{C.}},
\bauthor{\bsnm{Ehrhardt}, \binits{M.J.}}:
\bctitle{Convergence properties of a randomized primal-dual algorithm with applications to parallel mri}.
In: \bbtitle{International Conference on Scale Space and Variational Methods in Computer Vision},
pp. \bfpage{254}--\blpage{266}
(\byear{2021}).
\bcomment{Springer}
\end{bchapter}
\endbibitem

\bibitem[\protect\citeauthoryear{Sakurai and Napolitano}{2014}]{sakurai2014quantum}
\begin{botherref}
\oauthor{\bsnm{Sakurai}, \binits{J.}},
\oauthor{\bsnm{Napolitano}, \binits{J.}}:
Modern quantum mechanics. 2-nd edition.
Person New International edition
(2014)
\end{botherref}
\endbibitem

\bibitem[\protect\citeauthoryear{Pock and Chambolle}{2011}]{pock2011diagonal}
\begin{bchapter}
\bauthor{\bsnm{Pock}, \binits{T.}},
\bauthor{\bsnm{Chambolle}, \binits{A.}}:
\bctitle{Diagonal preconditioning for first order primal-dual algorithms in convex optimization}.
In: \bbtitle{2011 International Conference on Computer Vision},
pp. \bfpage{1762}--\blpage{1769}
(\byear{2011}).
\bcomment{IEEE}
\end{bchapter}
\endbibitem

\bibitem[\protect\citeauthoryear{Bauschke et~al.}{2011}]{bauschkeCombettes}
\begin{bbook}
\bauthor{\bsnm{Bauschke}, \binits{H.H.}},
\bauthor{\bsnm{Combettes}, \binits{P.L.}}, \betal:
\bbtitle{Convex Analysis and Monotone Operator Theory in Hilbert Spaces}
vol. \bseriesno{408}.
\bpublisher{Springer},
\blocation{~}
(\byear{2011})
\end{bbook}
\endbibitem

\bibitem[\protect\citeauthoryear{Bredies and Lorenz}{2018}]{bredieslorenz}
\begin{bbook}
\bauthor{\bsnm{Bredies}, \binits{K.}},
\bauthor{\bsnm{Lorenz}, \binits{D.}}:
\bbtitle{Mathematical Image Processing}.
\bpublisher{Springer},
\blocation{~}
(\byear{2018})
\end{bbook}
\endbibitem

\bibitem[\protect\citeauthoryear{Vetterli and Kovacevic}{1995}]{vetterli1995wavelets}
\begin{bbook}
\bauthor{\bsnm{Vetterli}, \binits{M.}},
\bauthor{\bsnm{Kovacevic}, \binits{J.}}:
\bbtitle{Wavelets and Subband Coding}.
\bpublisher{Prentice-hall},
\blocation{~}
(\byear{1995})
\end{bbook}
\endbibitem

\bibitem[\protect\citeauthoryear{Chaudhuri}{2001}]{chaudhuri2001super}
\begin{bbook}
\bauthor{\bsnm{Chaudhuri}, \binits{S.}}:
\bbtitle{Super-resolution Imaging}
vol. \bseriesno{632}.
\bpublisher{Springer},
\blocation{~}
(\byear{2001})
\end{bbook}
\endbibitem

\bibitem[\protect\citeauthoryear{Qu et~al.}{2015}]{quartz}
\begin{botherref}
\oauthor{\bsnm{Qu}, \binits{Z.}},
\oauthor{\bsnm{Richt{\'a}rik}, \binits{P.}},
\oauthor{\bsnm{Zhang}, \binits{T.}}:
Quartz: Randomized dual coordinate ascent with arbitrary sampling.
Advances in neural information processing systems
\textbf{28}
(2015)
\end{botherref}
\endbibitem

\bibitem[\protect\citeauthoryear{Robbins and Siegmund}{1971}]{robbins1971}
\begin{bchapter}
\bauthor{\bsnm{Robbins}, \binits{H.}},
\bauthor{\bsnm{Siegmund}, \binits{D.}}:
\bctitle{A convergence theorem for non negative almost supermartingales and some applications}.
In: \bbtitle{Optimizing Methods in Statistics},
pp. \bfpage{233}--\blpage{257}.
\bpublisher{Elsevier},
\blocation{~}
(\byear{1971})
\end{bchapter}
\endbibitem

\bibitem[\protect\citeauthoryear{Combettes and Pesquet}{2015}]{combettesPesquet}
\begin{barticle}
\bauthor{\bsnm{Combettes}, \binits{P.L.}},
\bauthor{\bsnm{Pesquet}, \binits{J.-C.}}:
\batitle{{Stochastic quasi-Fej{\'e}r block-coordinate fixed point iterations}}.
\bjtitle{SIAM Journal on Optimization}
\bvolume{25}(\bissue{2}),
\bfpage{1221}--\blpage{1248}
(\byear{2015})
\end{barticle}
\endbibitem

\bibitem[\protect\citeauthoryear{Pruessmann et~al.}{1999}]{sense}
\begin{barticle}
\bauthor{\bsnm{Pruessmann}, \binits{K.P.}},
\bauthor{\bsnm{Weiger}, \binits{M.}},
\bauthor{\bsnm{Scheidegger}, \binits{M.B.}},
\bauthor{\bsnm{Boesiger}, \binits{P.}}:
\batitle{Sense: sensitivity encoding for fast mri}.
\bjtitle{Magnetic Resonance in Medicine: An Official Journal of the International Society for Magnetic Resonance in Medicine}
\bvolume{42}(\bissue{5}),
\bfpage{952}--\blpage{962}
(\byear{1999})
\end{barticle}
\endbibitem

\bibitem[\protect\citeauthoryear{Ehrhardt et~al.}{2014}]{ehrhardt2014joint}
\begin{barticle}
\bauthor{\bsnm{Ehrhardt}, \binits{M.J.}},
\bauthor{\bsnm{Thielemans}, \binits{K.}},
\bauthor{\bsnm{Pizarro}, \binits{L.}},
\bauthor{\bsnm{Atkinson}, \binits{D.}},
\bauthor{\bsnm{Ourselin}, \binits{S.}},
\bauthor{\bsnm{Hutton}, \binits{B.F.}},
\bauthor{\bsnm{Arridge}, \binits{S.R.}}:
\batitle{Joint reconstruction of pet-mri by exploiting structural similarity}.
\bjtitle{Inverse Problems}
\bvolume{31}(\bissue{1}),
\bfpage{015001}
(\byear{2014})
\end{barticle}
\endbibitem

\bibitem[\protect\citeauthoryear{Ong et~al.}{2020}]{OngExtremeMRI}
\begin{barticle}
\bauthor{\bsnm{Ong}, \binits{F.}},
\bauthor{\bsnm{Zhu}, \binits{X.}},
\bauthor{\bsnm{Cheng}, \binits{J.Y.}},
\bauthor{\bsnm{Johnson}, \binits{K.M.}},
\bauthor{\bsnm{Larson}, \binits{P.E.}},
\bauthor{\bsnm{Vasanawala}, \binits{S.S.}},
\bauthor{\bsnm{Lustig}, \binits{M.}}:
\batitle{Extreme mri: Large-scale volumetric dynamic imaging from continuous non-gated acquisitions}.
\bjtitle{Magnetic resonance in medicine}
\bvolume{84}(\bissue{4}),
\bfpage{1763}--\blpage{1780}
(\byear{2020})
\end{barticle}
\endbibitem

\bibitem[\protect\citeauthoryear{Oscanoa et~al.}{2023}]{oscanoa2023coil}
\begin{botherref}
\oauthor{\bsnm{Oscanoa}, \binits{J.A.}},
\oauthor{\bsnm{Ong}, \binits{F.}},
\oauthor{\bsnm{Iyer}, \binits{S.S.}},
\oauthor{\bsnm{Li}, \binits{Z.}},
\oauthor{\bsnm{Sandino}, \binits{C.M.}},
\oauthor{\bsnm{Ozturkler}, \binits{B.}},
\oauthor{\bsnm{Ennis}, \binits{D.B.}},
\oauthor{\bsnm{Pilanci}, \binits{M.}},
\oauthor{\bsnm{Vasanawala}, \binits{S.S.}}:
Coil sketching for computationally-efficient mr iterative reconstruction.
arXiv preprint arXiv:2305.06482
(2023)
\end{botherref}
\endbibitem

\bibitem[\protect\citeauthoryear{Knoll et~al.}{2020}]{fastmri}
\begin{botherref}
\oauthor{\bsnm{Knoll}, \binits{F.}},
\oauthor{\bsnm{Zbontar}, \binits{J.}},
\oauthor{\bsnm{Sriram}, \binits{A.}},
\oauthor{\bsnm{Muckley}, \binits{M.J.}},
\oauthor{\bsnm{Bruno}, \binits{M.}},
\oauthor{\bsnm{Defazio}, \binits{A.}},
\oauthor{\bsnm{Parente}, \binits{M.}},
\oauthor{\bsnm{Geras}, \binits{K.J.}},
\oauthor{\bsnm{Katsnelson}, \binits{J.}},
\oauthor{\bsnm{Chandarana}, \binits{H.}}, et al.:
fastmri: A publicly available raw k-space and dicom dataset of knee images for accelerated mr image reconstruction using machine learning.
Radiology: Artificial intelligence
\textbf{2}(1)
(2020)
\end{botherref}
\endbibitem

\bibitem[\protect\citeauthoryear{Zbontar et~al.}{2018}]{fastmri_arxiv}
\begin{botherref}
\oauthor{\bsnm{Zbontar}, \binits{J.}},
\oauthor{\bsnm{Knoll}, \binits{F.}},
\oauthor{\bsnm{Sriram}, \binits{A.}},
\oauthor{\bsnm{Murrell}, \binits{T.}},
\oauthor{\bsnm{Huang}, \binits{Z.}},
\oauthor{\bsnm{Muckley}, \binits{M.J.}},
\oauthor{\bsnm{Defazio}, \binits{A.}},
\oauthor{\bsnm{Stern}, \binits{R.}},
\oauthor{\bsnm{Johnson}, \binits{P.}},
\oauthor{\bsnm{Bruno}, \binits{M.}}, et al.:
{fastMRI: An open dataset and benchmarks for accelerated MRI}.
arXiv preprint arXiv:1811.08839
(2018)
\end{botherref}
\endbibitem

\bibitem[\protect\citeauthoryear{Ong et~al.}{2022 \\ https://doi.org/10.5281/zenodo.5893788}]{sigpy}
\begin{botherref}
\oauthor{\bsnm{Ong}, \binits{F.}},
\oauthor{\bsnm{Martin}, \binits{J.}},
\oauthor{\bsnm{Grissom}, \binits{W.}},
\oauthor{\bsnm{Srinivasan}, \binits{S.}},
\oauthor{\bsnm{Johnson}, \binits{K.M.}},
\oauthor{\bsnm{Huynh}, \binits{C.}},
\oauthor{\bsnm{Desai}, \binits{A.}},
\oauthor{\bsnm{Li}, \binits{Z.}},
\oauthor{\bsnm{Tamir}, \binits{J.}},
\oauthor{\bsnm{Sandino}, \binits{C.}},
\oauthor{\bsnm{Shimron}, \binits{E.}},
\oauthor{\bsnm{Zeng}, \binits{D.}},
\oauthor{\bsnm{Mickevicius}, \binits{N.}}:
mikgroup/sigpy: {Minor} release to trigger {Zenodo} for {DOI}. (v0.1.24).
Zenodo
(2022 \\ https://doi.org/10.5281/zenodo.5893788)
\end{botherref}
\endbibitem

\bibitem[\protect\citeauthoryear{Pruessmann et~al.}{1998}]{pruessmann1998coil}
\begin{bchapter}
\bauthor{\bsnm{Pruessmann}, \binits{K.P.}},
\bauthor{\bsnm{Weiger}, \binits{M.}},
\bauthor{\bsnm{Scheidegger}, \binits{M.B.}},
\bauthor{\bsnm{Boesiger}, \binits{P.}}:
\bctitle{Coil sensitivity encoding for fast mri}.
In: \bbtitle{Proceedings of the ISMRM 6th Annual Meeting, Sydney},
vol. \bseriesno{1998}
(\byear{1998})
\end{bchapter}
\endbibitem

\bibitem[\protect\citeauthoryear{{Adler} et~al.}{2018}]{odl}
\begin{botherref}
\oauthor{\bsnm{{Adler}}, \binits{J.}},
\oauthor{\bsnm{{Kohr}}, \binits{H.}},
\oauthor{\bsnm{{Ringh}}, \binits{A.}},
\oauthor{\bsnm{{Moosmann}}, \binits{J.}},
\oauthor{\bsnm{{Sbanert}}},
\oauthor{\bsnm{{Ehrhardt}}, \binits{M.J.}},
\oauthor{\bsnm{{Lee}}, \binits{G.R.}},
\oauthor{\bsnm{{Niinimaki}}},
\oauthor{\bsnm{{Bgris}}},
\oauthor{\bsnm{{Verdier}}, \binits{O.}},
\oauthor{\bsnm{{Karlsson}}, \binits{J.}},
\oauthor{\bsnm{{Zickert}}},
\oauthor{\bsnm{{Palenstijn}}, \binits{W.J.}},
\oauthor{\bsnm{{{\"O}ktem}}, \binits{O.}},
\oauthor{\bsnm{{Chen}}, \binits{C.}},
\oauthor{\bsnm{{Loarca}}, \binits{H.A.}},
\oauthor{\bsnm{{Lohmann}}, \binits{M.}}:
{odlgroup/odl: ODL 0.7.0}.
Zenodo
(2018).
\doiurl{10.5281/zenodo.1442734}
\end{botherref}
\endbibitem

\bibitem[\protect\citeauthoryear{Beck and Teboulle}{2009}]{FISTA}
\begin{barticle}
\bauthor{\bsnm{Beck}, \binits{A.}},
\bauthor{\bsnm{Teboulle}, \binits{M.}}:
\batitle{A fast iterative shrinkage-thresholding algorithm for linear inverse problems}.
\bjtitle{SIAM Journal on Imaging Sciences}
\bvolume{2}(\bissue{1}),
\bfpage{183}--\blpage{202}
(\byear{2009})
\end{barticle}
\endbibitem

\end{thebibliography}

\end{document}